\documentclass{amsart}
\usepackage[normalem]{ulem}
\usepackage{amssymb,epsfig,mathrsfs,mathpazo,url}
\usepackage{color}
\usepackage{epsfig,bm}
\usepackage{booktabs}
\usepackage{graphicx}
\usepackage[multiple]{footmisc}
\usepackage{accents} 
\usepackage[T3,T1]{fontenc}
\DeclareSymbolFont{tipa}{T3}{cmr}{m}{n}
\DeclareMathAccent{\invbreve}{\mathalpha}{tipa}{16}

\newtheorem{theorem}{Theorem}[section]
\newtheorem{lemma}[theorem]{Lemma}
\newtheorem{proposition}[theorem]{Proposition}
\newtheorem{assumption}[theorem]{Assumption}

\newtheorem{algorithm}[theorem]{Algorithm}

\theoremstyle{definition}
\newtheorem{definition}[theorem]{Definition}

\theoremstyle{remark}
\newtheorem{remark}[theorem]{Remark}

\numberwithin{equation}{section}

\newcommand{\eps}{\varepsilon}

\newcommand{\R}{\mathbb R}

\newcommand{\N}{\mathbb N}

\newcommand{\cL}{\mathcal L}
\newcommand{\cS}{\mathcal S}
\newcommand{\Lis}{\cL\mathrm{is}}

\newcommand{\identity}{\mathrm{Id}}

\DeclareMathOperator{\ran}{ran}
\DeclareMathOperator{\supp}{supp}
\DeclareMathOperator{\clos}{clos}

\DeclareMathOperator{\blockdiag}{blockdiag}
\DeclareMathOperator{\diam}{diam}
\DeclareMathOperator*{\argmin}{argmin}
\DeclareMathOperator{\dist}{dist}

\DeclareMathOperator{\essinf}{ess\,inf}

\DeclareMathOperator{\Span}{span}

\newcommand{\new}[1]{{\color{black}{#1}}}

\newcommand{\comment}[1]{{\color{blue}{}}}

\newcommand{\be}{\begin{equation}}
\newcommand{\ee}{\end{equation}}

\newcommand{\1}{\mathbb 1}

\newcommand{\bbT}{\mathbb{T}}
\newcommand{\tria}{{\mathcal T}}

\newcommand{\uumlaut}{{\"u}}

\newcommand{\gen}{{\tt gen}}

\DeclareRobustCommand{\ubar}[1]{\underaccent{\bar}{#1}}
\newcommand{\udelta}{{\ubar{\delta}}}
\newenvironment{algotab}%
{\par\begin{samepage}%
\begin{tabbing}\ttfamily%
 \hspace*{5mm}\=\hspace{3ex}\=\hspace{3ex}\=\hspace{3ex}\=\hspace{3ex}%
\=\hspace{3ex}\=\hspace{3ex}\=\hspace{3ex}\=\hspace{3ex}\kill}%
{\end{tabbing}\end{samepage}}

\title[An adaptive method for parabolic evolution equations]{A wavelet-in-time, finite element-in-space adaptive method for parabolic evolution equations}

\date{\today}

\author{Rob Stevenson, Raymond van Veneti\"{e}, Jan Westerdiep}

\address{
Korteweg--de Vries (KdV) Institute for Mathematics, University of Amsterdam, P.O. Box 94248, 1090 GE Amsterdam, The Netherlands.
}
\email{r.p.stevenson@uva.nl, j.h.westerdiep@uva.nl, R.vanVenetie@uva.nl}

\thanks{The second and third author have been supported by the Netherlands Organization for Scientific Research (NWO) under contract.~no.~613.001.652}



\begin{document}

\begin{abstract} In this work, an $r$-linearly converging adaptive solver is constructed for parabolic evolution equations in a simultaneous space-time variational formulation. Exploiting the product structure of the space-time cylinder, the family of trial spaces that we consider are given as the spans of wavelets-in-time and (locally refined) finite element spaces-in-space. Numerical results illustrate our theoretical findings.
\end{abstract}

\maketitle
\section{Introduction}
This paper is about the adaptive numerical solution of parabolic evolution equations written in a simultaneous space-time variational formulation. In comparison to the usually applied time-marching schemes, simultaneous space-time solvers offer the following potential advantages:
\begin{itemize}
\item local, adaptive refinements simultaneous in space and time (\cite{249.4,243.867,75.38}),
\item quasi-best approximation from the selected trial space (`Cea's lemma')  (\cite{11,169.055,249.99}), being a necessary requirement
for proving optimal rates for adaptive routines (\cite{38.4,171.66,243.867}),
\item superior parallel performance (\cite{70.2,234.7,169.06,306.65}),
\item using the product structure of the space-time cylinder,  sparse tensor product approximation (\cite{76.21,38.4,171.66,243.867}) which allows to solve the whole time evolution at a complexity of solving the corresponding stationary problem.
\end{itemize}
Other relevant publications on space-time solvers include \cite{249.2, 169.05, 249.3,64.5765,64.585}.

In any case without applying sparse tensor product approximation, a disadvantage of the space-time approach is the larger memory consumption because instead of solving a sequence of PDEs on a $d$-dimensional space, one has to solve one PDE on a $(d+1)$-dimensional space. This disadvantage, however, disappears when one needs simultaneously the whole time evolution as for example with problems of optimal control (\cite{77.5,19.96}) or data-assimilation (\cite{58.6}).

\subsection{Parabolic problem in a simultaneous space-time variational formulation}
 For some separable Hilbert spaces $V \hookrightarrow H$ with dense embedding (e.g.~$H_0^1(\Omega)$ and $L_2(\Omega)$ for the model problem of the heat equation on a spatial domain $\Omega \subset \R^d$), and a boundedly invertible $A(t)=A(t)'\colon V \rightarrow V'$ with $(A(t)\cdot)(\cdot) \eqsim \|\cdot\|_V^2$ (e.g.~ $(A(t)\eta)(\zeta)=\int_\Omega \nabla \eta \cdot \nabla \zeta\,d{\bf x}$), we consider
$$
\left\{
\begin{array}{rl}
\frac{d u}{d t}(t) +A(t) u(t)&\!\!\!= g(t) \quad(t \in (0,T)),\\
u(0) &\!\!\!= u_0.
\end{array}
\right.
$$
An application of a variational formulation of the PDE over space and time leads to an equation
\be \label{I1}
\left[\begin{array}{@{}c@{}} B \\ \gamma_0\end{array} \right] u=\left[\begin{array}{@{}c@{}} g \\ u_0\end{array} \right]
\ee
where, with $X:=L_2(I;{V}) \cap H^1(I;V')$ and $Y:=L_2(I;{V})$,
the operator at the left hand side is  boundedly invertible $X \rightarrow Y'\times H$.

\subsection{Our previous work}
In \cite{38.4,243.867} we equipped $X$ and $Y$ with Riesz bases being tensor products of wavelet bases in space in time, and $H$ with some spatial Riesz basis. Consequently, the equation \eqref{I1} got an equivalent formulation as a bi-infinite well-posed matrix-vector equation $\left[\begin{array}{@{}c@{}} {\bf B} \\ \bm{\gamma}_0\end{array} \right] {\bf u} =\left[\begin{array}{@{}c@{}} {\bf g} \\ {\bf u}_0\end{array} \right]$
(actually, in \cite{243.867}, we considered a formulation of first order, and in \cite{38.4} we used a variational formulation with essentially interchanged roles of $X$ and $Y$, which however is irrelevant for the current discussion).
To get a coercive bilinear form we formed normal equations to which we applied an adaptive wavelet scheme (\cite{45.2}). With such a scheme the norm of a sufficiently accurate approximation of the (infinite) residual vector of a current approximation is used as an a posteriori error estimator. The coefficients in modulus of this vector are applied as local error indictors in a bulk chasing \new{(or D\"{o}rfler marking)} procedure. The resulting adaptive algorithm converges at the best possible rate in linear computational complexity.

The goal of the current work is to investigate to what extent similar optimal theoretical results can be shown
\new{when we replace the wavelets-in-space by finite element spaces. In any case for locally refined partitions, a disadvantage will be that these spaces cannot be equipped with bases that are uniformly `stable' w.r.t.~a range of Sobolev norms. On the other hand, the application of finite elements has quantitative advantages, as provided by a very localized single-scale basis and efficient multigrid preconditioners, and it avoids a cumbersome construction of wavelet bases on domains that are not hyperrectangles.}

\subsection{Least squares minimization} Without having Riesz bases for $X$ and $Y$, already the step of first discretizing and then forming normal equations does not apply, and we reverse their order.
A problem equivalent to \eqref{I1} is to compute
\be \label{I2}
u=\argmin_{w \in X} \|Bw -g\|_{Y'}^2+\|\gamma_0 w-u_0\|_H^2.
\ee
An obvious approach for the numerical approximation is to consider the minimization over finite dimensional subspaces $X^\delta$ of $X$, which however is not feasible because of the presence of the dual norm.

For trial spaces $X^\delta$ that are `full' (or `sparse') tensor products of finite element spaces in space and time, in \cite{11}
it was shown how to construct corresponding test spaces $Y^\delta \subset Y$ of similar type and dimension, such that $(X^\delta,Y^\delta)$ is uniformly \emph{inf-sup stable} meaning that
 when the continuous dual norm $\|\cdot\|_{Y'}$ is replaced by the discrete dual norm $\|\cdot\|_{{Y^\delta}'}$, a minimization over $X^\delta$ yields a quasi-best approximation to $u$ from $X^\delta$.
Such a family of trial spaces however does not allow to create a nested sequence of trial spaces by adaptive \emph{local} refinements.

\subsection{Family of inf-sup stable pairs of trial and test spaces} To construct an alternative, \new{considerably} larger family, let $\Sigma$ be a wavelet Riesz basis for $L_2(0,T)$ that, after renormalization, is also a Riesz basis for $H^1(0,T)$.
We equip this basis with a tree structure where every wavelet that is not on the coarsest level has a parent on the next coarser level.
In space, we consider the collection of all linear finite element spaces that can be generated by conforming newest vertex bisection starting from an initial conforming partition of a polytopal  $\Omega$ into $d$-simplices.
The restriction to linear finite elements is not essential and is made for simplicity only.
Now we consider trial spaces $X^\delta$ that are spanned by a number of wavelets each of them tensorized with a finite element space from the aforementioned collection. In order to be able to apply the arising system matrices in linear complexity (\cite{171.7,306.6}) we impose the condition that if a wavelet tensorized with a finite element space is in the spanning set, then so is its parent wavelet tensorized with a finite element space that includes the former one.

The infinite collection of finite element spaces can be associated \new{with} a hierarchical `basis' that can be equipped with a tree structure. Each hierarchical basis function, except those on the coarsest level, is associated \new{with} a node $\nu$ that was inserted as the midpoint of an edge connecting two nodes on the next coarser level, which nodes we call the parents of $\nu$. With this definition there is a one-to-one correspondence between the finite element spaces from
our collection and the spans of the sets of hierarchical basis functions that form trees.
Consequently, our collection of trial spaces $X^\delta$ consists of the spans of sets of tensor products of wavelets-in-time and hierarchical basis functions-in-space which sets are \emph{downwards closed}, also known as \emph{lower}, in the sense that if a pair of a wavelet and a hierarchical basis function is in the set, then so are all its parents in time and space.
Spaces from this collection can be `locally' expanded by adding the span of a tensor product of a wavelet and hierarchical basis function one-by-one.

For this family of spaces  $X^\delta$ we construct a corresponding family of spaces $Y^\delta \subset Y$ of similar type such that
each pair $(X^\delta,Y^\delta)$ is uniformly inf-sup stable, with the dimension of $Y^\delta$ being proportional to that of $X^\delta$.
Furthermore, using the properties of the wavelets in time and by applying multigrid preconditioners in space we construct optimal preconditioners at $X$ and $Y$-side which allow a fast solution of the discrete problems $\argmin_{w \in X^\delta} \|Bw -g\|_{{Y^\delta}'}^2+\|\gamma_0 w-u_0\|_H^2$.

\subsection{Adaptive algorithm} Having fixed the family of trial spaces, it remains to develop an algorithm that selects a suitable, preferably quasi-optimal, nested sequence of spaces from the family adapted to the solution $u$ of  \eqref{I2}.
The theory \new{on} adaptive \new{(Ritz--)Galerkin} approximations for such quadratic minimization problems is in a mature state. As noticed before, however, \new{straightforward} Galerkin approximations for  \eqref{I2}, i.e., \new{approximations obtained by replacing $X$ by $X^\delta$} are not computable \new{due to the presence of the dual norm.}

Therefore given $X^\delta$, let $X^\udelta \supset X^\delta$ be such that saturation holds, i.e., for some constant $\zeta<1$, it holds that $\inf_{w \in X^\udelta}\|u-w\|_X \leq \zeta \inf_{w \in X^\delta}\|u-w\|_X$. We now replace problem \eqref{I2} by
\be \label{I3}
u^{\udelta \udelta}=\argmin_{w \in X^\udelta} \|Bw -g\|_{{Y^\udelta}'}^2+\|\gamma_0 w-u_0\|_H^2,
\ee
where in the notation $u^{\udelta \udelta}$ the first instance of $\udelta$ refers to the space $Y^\udelta$ and the second to the space $X^\udelta$. Its (computable) Galerkin approximation from $X^\delta$ is given by
$$
u^{\udelta \delta}=\argmin_{w \in X^\delta} \|Bw -g\|_{{Y^\udelta}'}^2+\|\gamma_0 w-u_0\|_H^2.
$$
By a standard adaptive procedure, described below, we expand $X^\delta$ to some $X^{\tilde{\delta}} \subseteq X^\udelta$ such that $u^{\udelta \tilde{\delta}}$ is \new{nearer} to $u^{\udelta \udelta}$ than $u^{\udelta \delta}$ \new{is}.
Next, we replace $Y^\udelta$ by $Y^{\ubar{\tilde{\delta}}}$ (being the test space corresponding to $X^{\ubar{\tilde{\delta}}}$) and repeat (i.e.~consider \eqref{I3} with $(\udelta,\udelta)$ reading as $(\ubar{\tilde{\delta}},\ubar{\tilde{\delta}})$, and improve its Galerkin approximation $u^{\ubar{\tilde{\delta}} \tilde{\delta}}$ from $X^{\tilde{\delta}}$ by an adaptive enlargement of the latter space).

The adaptive expansion of the trial space $X^\delta$ to $X^{\tilde{\delta}}$ will be by the application of the usual solve-estimate-mark-refine paradigm, where
the error indicators are \new{given by} the coefficients of the residual vector w.r.t.~(modified) tensor product basis functions that were added to $X^\delta$ to create $X^\udelta$. In order for this collection of additional tensor product basis functions to be stable in $X$-norm, for this step we modify the hierarchical basis functions such that they get a vanishing moment, and therefore become closer to `real' wavelets.

Under the aforementioned saturation assumption, we prove that the overall adaptive procedure produces an $r$-linearly converging sequence to the solution.

\subsection{Numerical results} \new{We tested} the adaptive algorithm in several examples with a two-dimensional spatial domain. In all but one case, we observed a convergence rate equal to $1/2$, being the best that can be expected in view of the piecewise polynomial degree of the trial functions \emph{and} the tensor product construction, and for non-smooth solutions improving upon usual non-adaptive approximation.
Only for the case where $u_0=1$ and homogenous Dirichlet boundary conditions are prescribed, the observed rate was reduced to $0.4$. It is unknown whether or not this is the best non-linear approximation rate for our family of trial spaces.

Thanks to the use of optimal preconditioners and that of a carefully designed matrix-vector multiplication routine, which generalizes such a routine for sparse-grids introduced in \cite{18.83} to adaptive settings, we observe that
throughout the whole execution of the adaptive loop the total runtime remains proportional to the current number of unknowns. \comment{Paragraph deleted}

\subsection{Organization} This paper is organized as follows: In Sect.~\ref{S2} the well-posed space-time variational formulation of the parabolic problem is discussed, and in Sect.~\ref{S3} we discuss its inf-sup stable discretisation.
 The adaptive solution procedure is presented in Sect.~\ref{sec:adaptive}, and its convergence is proven.
The construction of the trial and test spaces is detailed in Sect.~\ref{Srealization}, and optimal preconditioners are presented.
In Sect.~\ref{S6}, the definition of the enlarged space $X^{\udelta}$ is given, and the construction of a stable basis of a stable complement space of $X^{\delta}$ in $X^{\udelta}$ is outlined. Numerical results are presented in Sect.~\ref{Snumerics}, and a conclusion is formulated in Sect.~\ref{Sconclusion}.

\subsection{Notations}\label{ssec:1.2}
In this work, by $C \lesssim D$ we will mean that $C$ can be bounded by a multiple of $D$, independently of parameters which \new{$C$} and \new{$D$} may depend on.
Obviously, $C \gtrsim D$ is defined as $D \lesssim C$, and $C\eqsim D$ as $C\lesssim D$ and $C \gtrsim D$.

For normed linear spaces $E$ and $F$, by $\cL(E,F)$ we will denote the normed linear space of bounded linear mappings $E \rightarrow F$,
and by $\Lis(E,F)$ its subset of boundedly invertible linear mappings $E \rightarrow F$.
We write $E \hookrightarrow F$ to denote that $E$ is continuously embedded into $F$.
For simplicity only, we exclusively consider linear spaces over the scalar field $\R$.

\section{Space-time formulations of a parabolic evolution problem} \label{S2}
Let $V,H$ be separable Hilbert spaces of functions on some ``spatial domain''
such that $V \hookrightarrow H$ with dense embedding.
Identifying $H$ with its dual, we obtain the Gelfand triple
$V \hookrightarrow H \simeq H' \hookrightarrow V'$.

For a.e.
$$
t \in I:=(0,T),
$$
let $a(t;\cdot,\cdot)$ denote a bilinear form on $V \times V$ such that for
any $\eta,\zeta \in V$, $t \mapsto a(t;\eta,\zeta)$ is measurable on $I$,
and such that for some $\varrho \in \R$, for
a.e. $t\in I$,
\begin{alignat}{3} \label{1}
|a(t;\eta,\zeta)|
& \lesssim  \|\eta\|_{V} \|\zeta\|_{V} \quad &&(\eta,\zeta \in V) \quad &&\text{({\em boundedness})},
\\ \label{2}
 a(t;\eta,\eta) +\varrho\langle \eta,\eta\rangle_H &\gtrsim \|\eta\|_{V}^2 \quad
&&(\eta \in {V}) \quad &&\text{({\em G{\aa}rding inequality})}.
\end{alignat}

With $A(t) \in \Lis({V},V')$ being defined by $ (A(t) \eta)(\zeta)=a(t;\eta,\zeta)$, given a forcing function $g$ and an initial value $u_0$, we are interested in solving the {\em parabolic initial value problem} to finding $u$ such that
\begin{equation} \label{11}
\left\{
\begin{array}{rl}
\frac{d u}{d t}(t) +A(t) u(t)&\!\!\!= g(t) \quad(t \in I),\\
u(0) &\!\!\!= u_0.
\end{array}
\right.
\end{equation}

In a simultaneous space-time variational formulation, the parabolic PDE reads as finding $u$ from a suitable space of functions of time and space such that
$$
(Bw)(v):=\int_I
\langle { \textstyle \frac{d w}{dt}}(t), v(t)\rangle +
a(t;w(t),v(t)) dt = \int_I
\langle g(t), v(t)\rangle =:g(v)
$$
for all $v$ from another suitable space of functions of time and space.
One  possibility to enforce the initial condition is by testing it against additional test functions. A proof of the following result can be found in \cite{247.15}, cf.
\cite[Ch.XVIII, \S3]{63} and \cite[Ch.~IV, \S26]{314.9} for slightly different statements.

\begin{theorem} \label{thm0} With $X:=L_2(I;{V}) \cap H^1(I;V')$, $Y:=L_2(I;{V})$,
under conditions \eqref{1} and \eqref{2} it holds that
$$
\left[\begin{array}{@{}c@{}} B \\ \gamma_0\end{array} \right]\in \Lis(X,Y' \times H),
$$
where for $t \in \bar{I}$, $\gamma_t\colon u \mapsto u(t,\cdot)$ denotes the trace map.
That is, assuming $g \in Y'$ and $u_0 \in H$, finding $u \in X$ such that
\be \label{x12}
\left[\begin{array}{@{}c@{}} B \\ \gamma_0\end{array} \right] u=\left[\begin{array}{@{}c@{}} g \\ u_0\end{array} \right]
\ee
is a well-posed simultaneous space-time variational formulation of \eqref{11}.
\end{theorem}

With $\tilde{u}(t):=u(t) e^{-\varrho t}$, \eqref{11} is equivalent to $
\frac{d \tilde{u}}{d t}(t) +(A(t)+\varrho \identity) \tilde{u}(t)= g(t)e^{-\varrho t}$  ($t \in I$),
$\tilde{u}(0) = u_0$. Since $((A(t)+\varrho \identity)\eta)(\eta) \gtrsim \|\eta\|_{V}^2$, w.l.o.g.~we assume that
\eqref{2} is valid for $\varrho=0$, i.e., $a(t;\cdot,\cdot)$ is \emph{coercive} uniformly for a.e.~$t \in I$.

For simplicity, additionally we assume that $a(t;\cdot,\cdot)$ is \emph{symmetric}, and define
$A=A' \in \Lis(Y,Y')$ by $(A w)(v)=\int_I (A(t) w(t))v(t) \,dt$.

Because $\left[\begin{array}{@{}cc@{}} A & 0 \\ 0 & \identity \end{array}\right] \in \Lis(Y\times H,Y'\times H)$, an equivalent formulation of \eqref{x12} as a self-adjoint saddle point equation reads as finding $(\mu,\sigma,u) \in Y\times H\times X$ (where $\mu=0=\sigma$) such that
\begin{align} \label{m0}
\left[\begin{array}{@{}ccc@{}} A & 0 & B\\ 0 & \identity & \gamma_0\\ B' & \gamma_0' & 0\end{array}\right]
\left[\begin{array}{@{}c@{}} \mu \\ \sigma \\ u \end{array}\right]&=
\left[\begin{array}{@{}c@{}} g \\ u_0 \\ 0 \end{array}\right],
\intertext{or equivalently}
 \label{m1}
\left[\begin{array}{@{}cc@{}} A & B\\ B' & -\gamma_0' \gamma_0\end{array}\right]
\left[\begin{array}{@{}c@{}} \mu  \\ u \end{array}\right]&=
\left[\begin{array}{@{}c@{}} g \\ -\gamma_0' u_0 \end{array}\right],
\intertext{or} \label{m2}
\underbrace{(B' A^{-1} B+\gamma_0'\gamma_0)}_{S:=}u&=\underbrace{B' A^{-1}g+\gamma_0' u_0}_{f:=}.
\end{align}
\new{By employing $H \simeq H'$, here $\gamma_0' \in \cL(H,X')$ is defined by $(\gamma_0' v)(w):=\langle v,\gamma_0 w\rangle_H$.}

We equip $Y$ and $X$ with \emph{`energy'-norms}
$$
\|\cdot\|_Y^2:=(A\cdot)(\cdot),\quad \|\cdot\|_X^2:=\|\cdot\|_Y^2+\|\partial_t \cdot\|_{Y'}^2+\|\gamma_T \cdot\|_H^2,
$$
which are equivalent to the canonical norms on $Y$ and $X$. Notice that \eqref{m0}--\eqref{m2} are the Euler--Langrange equations that result from the
\new{least squares problem \eqref{I2}.}

\begin{lemma} \label{100} We have $\|\cdot\|_X^2=(S\cdot)(\cdot)$.
\end{lemma}

\begin{proof}
It holds that
\begin{align*}
\|w\|_X^2&=\sup_{0 \neq v_1 \in Y} \frac{(B w)(v_1)}{\|v_1\|_Y^2}+\|\gamma_0 w\|_H^2
=\sup_{0 \neq (v_1,v_2) \in Y \times H} \frac{((B w)(v_1)+\langle \gamma_0 w,v_2\rangle_H)^2}{\|v_1\|_Y^2+\|v_2\|_H^2}\\
&=(Sw)(w),
\end{align*}
  \new{where the first equality is found} in e.g.~\cite[Thm.~2.1]{70.95}, and, \new{seeing} that
$S$ is the Schur complement of the operator in \eqref{m0}, the last one in e.g.~\cite[Lem.~2.2]{174}.
\end{proof}

\section{Discretizations} \label{S3}
\subsection{Galerkin discretization of the Schur complement equation}
Let $(X^\delta)_{\delta \in \Delta}$ be a collection of closed, e.g, finite dimensional, subspaces of $X$, so equipped with $\|\cdot\|_X$.
We will specify such a family in Sect.~\ref{Srealization}-\ref{Smapping}.
We define a \emph{partial order} on $\Delta$ by
$$
\delta \preceq \tilde{\delta} \Longleftrightarrow X^\delta \subseteq X^{\tilde{\delta}}.
$$
For $\delta \in \Delta$, let $u_\delta \in X^\delta$ denote the {\em Galerkin approximation to the solution $u$ of \eqref{m2}}, i.e., the solution of
\be \label{50}
(Su_\delta)(v)=f(v) \quad (v \in X^\delta),
\ee
being the best approximation to $u$ from $X^\delta$ w.r.t.~$\|\cdot\|_X$.
\new{This $u_\delta$ is the solution of \eqref{I2} with the minimization being performed over $X^\delta$ instead of over $X$.}

For proving convergence of an adaptive solution routine, as well as for a posteriori error estimation, we shall make the following assumption.
\begin{assumption}[Saturation] \label{saturation}
There exists a collection of subspaces $({}^{\delta\!}G \times {}^{\delta\!}U_0)_{\delta \in \Delta} \subseteq Y' \times H$,
a mapping $\underline{\cdot}\colon \Delta \rightarrow \Delta\colon \delta \mapsto \udelta$ where $\udelta\succeq \delta$,
and
some fixed constant $\zeta<1$ such that for all $\delta \in \Delta$, assuming that $(g,u_0) \in {}^{\delta\!}G \times {}^{\delta\!}U_0$,
\be \label{reduction}
\|u-u_\udelta\|_X \leq \zeta \|u-u_\delta\|_X.
\ee
\end{assumption}

\begin{remark} \label{data-oscillation}
  Notice that above assumption cannot be valid without restrictions on the right-hand side $f =B' A^{-1}g+\gamma_0' u_0\in X'$.
  Indeed \new{for} any $X^\delta \subset X^{\udelta} \subsetneq X$, consider a non-zero $f \in X'$ that vanishes on $X^{\udelta}$. Then $u_\delta=u_\udelta=0 \neq u$, \new{violating~\eqref{reduction}}.
\end{remark}

\begin{center}
  \parbox{11cm}{\emph{For \new{now}, we will operate under the restrictive assumption that whenever we apply \eqref{reduction} (visible by the appearance of the constant $\zeta$) we simply assume that
$(g,u_0) \in {}^{\delta\!}G \times {}^{\delta\!}U_0$. Later, in Sect.~\ref{Soscillation}, we will remove this assumption.}}
\end{center}
\smallskip

The discretized problem from \eqref{50} only serves theoretical purposes. Indeed, since the Schur complement operator $S$ contains the inverse of $A$, there is no way to determine $u_\delta$ exactly. The reason to introduce \eqref{50} is that $S$ is an \emph{elliptic} operator, so that for $\delta \preceq \tilde{\delta}$ we can make use of $\|u-u_{\tilde{\delta}}\|_X^2=\|u-u_{\delta}\|_X^2-\|u_{\tilde{\delta}}-u_{\delta}\|_X^2$, being a crucial tool for proving convergence of adaptive algorithms.

\subsection{Uniformly stable Galerkin discretization of the saddle-point formulation}
Our numerical approximations will be based on Galerkin discretizations of the saddle-point formulation \eqref{m1}.
Let $(Y^\delta)_{\delta \in \Delta}$ be a collection of closed subspaces of $Y$, so equipped with $\|\cdot\|_Y$,  such that
\be \label{x20}
X^\delta \subseteq Y^\delta \quad (\delta \in \Delta),
\ee
and
\be \label{x21}
1 \geq \gamma_\Delta:=\inf_{\delta \in \Delta}\inf_{0 \neq w \in X^\delta}\sup_{0\neq v \in Y^\delta} \frac{(\partial_t w)(v)}{\|\partial_t w\|_{Y'}\|v\|_Y} >0.
\ee
Notice that $1-\gamma_\Delta$ can be made arbitrarily small by selecting, for each $\delta\in \Delta$, $Y^\delta$ sufficiently large in relation to $X^\delta$.

For $\delta, \hat{\delta} \in \Delta$ with $Y^{\hat{\delta}} \supseteq Y^\delta$, and
$E_Y^{\hat{\delta}}$, $E_X^\delta$ denoting the embeddings $Y^{\hat{\delta}} \rightarrow Y$, $X^\delta \rightarrow X$,
let $(\mu^{\hat{\delta} \delta},u^{\hat{\delta} \delta}) \in Y^{\hat{\delta}} \times X^\delta$ be the solution of
\be \label{m8}
\left[\begin{array}{@{}ccc@{}}{E_Y^{\hat{\delta}}}' A E_Y^{\hat{\delta}}& {E_Y^{\hat{\delta}}}' B E^\delta_X\\ {E^\delta_X}' B' E_Y^{\hat{\delta}}& -{E^\delta_X}' \gamma_0' \gamma_0 E^\delta_X \end{array}\right]
\left[\begin{array}{@{}c@{}} \mu^{\hat{\delta} \delta} \\ u^{\hat{\delta} \delta} \end{array}\right]=
\left[\begin{array}{@{}c@{}} {E^{\hat{\delta}}_Y}' g \\ -{E^\delta_X}' \gamma_0' u_0 \end{array}\right],
\ee
or, equivalently,
\be \label{m9}
\underbrace{{E^\delta_X}'(B' E^{\hat{\delta}}_Y({E^{\hat{\delta}}_Y}' A E^{\hat{\delta}}_Y)^{-1} {E^{\hat{\delta}}_Y}' B+\gamma_0'\gamma_0)E^\delta_X}_{S^{\hat{\delta} \delta}:=}u^{\hat{\delta} \delta}=\underbrace{{E^\delta_X}' (B' E^{\hat{\delta}}_Y({E^{\hat{\delta}}_Y}' A E^{\hat{\delta}}_Y)^{-1} {E^{\hat{\delta}}_Y}' g+\gamma_0' u_0)}_{f^{\hat{\delta} \delta}:=}.
\ee
Below we \new{show} that \eqref{m8}--\eqref{m9} are uniquely solvable.
\new{In} `operator language', \eqref{m8} is the Galerkin discretization of \eqref{m1} on the closed subspace $Y^{\hat{\delta}} \times X^\delta \subseteq Y \times X$.
Unless $Y^{\hat{\delta}}=Y$, it holds that $S^{\hat{\delta} \delta} \neq {E_X^\delta}' S E_X^\delta$ and $f^{\hat{\delta} \delta} \neq {E_X^\delta}'f$,
so generally $u^{\hat{\delta} \delta} \neq u_\delta$.

As we will see, however, for $Y^\delta$, and thus $Y^{\hat{\delta}}$, `large' in relation to $X^\delta$, $u^{\hat{\delta} \delta}$ will be `close' to $u_\delta$. This will allow us to show that ($r$-linear) convergence of a sequence of Galerkin solutions $u_\delta$ of \eqref{50} implies \new{that} of the corresponding sequence $u^{\hat{\delta} \delta}$.

We equip $X^\delta$ with a family of `energy' norms
$$
\|w\|^2_{X^{\hat{\delta}\delta}}:=\|w\|^2_{Y}+\sup_{0 \neq v \in Y^{\hat{\delta}}} \frac{(\partial_t w)(v)^2}{\|v\|_Y^2}+\new{\|\gamma_T w\|^2_H}.
$$
By definition of $\gamma_\Delta$ it holds that
\be \label{m19}
\gamma_\Delta\|\cdot\|_{X} \leq \|\cdot\|_{X^{\hat{\delta}\delta}} \leq \|\cdot\|_{X} \quad\text{on } X^\delta.
\ee

\new{Similar} to Lemma~\ref{100} we have the following result.

\begin{lemma}[{\cite[\new{Lem.~3.4}]{249.99}}] \label{101}
Thanks to \eqref{x20} (and $Y^\delta \subseteq Y^{\hat{\delta}}$), for $w \in X^\delta$, \new{we have}
$$
\|w\|_{X^{\hat{\delta}\delta}}^2=(S^{\hat{\delta}\delta}w)(w)=\sup_{0 \neq v \in Y^{\hat{\delta}} } \frac{(B w)(v)^2}{\|v\|_Y^2}+\|\gamma_0 w\|_H^2.
$$
\end{lemma}
\noindent By using additionally \eqref{x21} this result shows that $(S^{\hat{\delta}\delta}\cdot)(\cdot)$ is coercive on $X^\delta \times X^\delta$ so that \eqref{m9}, and thus \eqref{m8}, has a unique solution.

Moreover, we have the following result.

\begin{theorem}[{\cite[Thm.~3.7]{249.99}}] Thanks to \eqref{x20} (and $Y^\delta \subseteq Y^{\hat{\delta}}$) and \eqref{x21}, it holds that
\be \label{m6}
\|u-u_\delta\|_X \leq \|u-u^{\hat{\delta}\delta}\|_X \leq \gamma_\Delta^{-1} \|u-u_\delta\|_X.
\ee
\end{theorem}\comment{Remark deleted}

\subsection{Modified discretized saddle-point} \label{Smod}
In view of obtaining an efficient implementation, in the definition of $(\mu^{\hat{\delta} \delta},u^{\hat{\delta} \delta})$ in \eqref{m8}, and so in that of
$S^{\hat{\delta} \delta}$ and $f^{\hat{\delta} \delta}$ in \eqref{m9}, we replace $({E_Y^{\hat{\delta}}}' A E_Y^{\hat{\delta}})^{-1}$
by some $K_Y^{\hat{\delta}}={K_Y^{\hat{\delta}}}' \in \Lis({Y^{\hat{\delta}}}',Y^{\hat{\delta}})$
for which both
\be \label{kappadelta}
\frac{((K_Y^{\hat{\delta}})^{-1} v)(v)}{(A v)(v)} \in [\kappa_\Delta^{-1},\kappa_\Delta] \quad(\delta \in \Delta,\, v \in Y^{\hat{\delta}})
\ee
\new{for some constant $\kappa_\Delta \geq 1$}
(i.e.~$K_Y^{\hat{\delta}}$ is an optimal (self-adjoint and coercive) \emph{preconditioner} for ${E_Y^{\hat{\delta}}}' A E_Y^{\hat{\delta}}$), and which can be applied at linear cost.
The resulting system \eqref{m9} is now amenable to the application of the (preconditioned) conjugate \new{gradient} iteration.
Despite this modification, we keep using the old notations for $\mu^{\hat{\delta} \delta}$, $u^{\hat{\delta} \delta}$,
$S^{\hat{\delta} \delta}$, \mbox{$\|\cdot\|_{X^{\hat{\delta}\delta}}:=(S^{\hat{\delta}\delta}\cdot)(\cdot)^{\frac12}$}, and $f^{\hat{\delta} \delta}$.

As shown in \cite[Remark~3.8]{249.99}, instead of \eqref{m6} now it holds that
\be  \label{m6b}
\|u-u_\delta\|_X \leq \|u-u^{\hat{\delta}\delta}\|_X \leq \frac{\kappa_\Delta}{\gamma_\Delta} \|u-u_\delta\|_X,
\ee
whereas one deduces that \eqref{m19} now should be read as
\be \label{m19b}
\frac{\gamma_\Delta}{\sqrt{\kappa_\Delta}} \|\cdot\|_{X} \leq \|\cdot\|_{X^{\hat{\delta}\delta}} \leq \sqrt{\kappa_\Delta} \|\cdot\|_{X} \quad\text{on } X^\delta.
\ee
For our forthcoming analysis, we will need $\frac{\kappa_\Delta}{\gamma_\Delta}-1$ to be sufficiently small.

\begin{remark} \label{x14}
Later, in the proof of Proposition~\ref{alternative}, temporarily we will consider the system \eqref{m8} with $\hat{\delta}=\delta$ (i.e.~$Y^{\hat{\delta}}=Y^\delta$), but with $X^\delta$ replaced by $X$, and, as we do in the current subsection, ${E_Y^\delta}' A E_Y^\delta$ replaced by $(K_Y^{\delta})^{-1}$. The resulting Schur operator $B' E_Y^\delta K_Y^\delta {E_Y^\delta}' B +\gamma_0' \gamma_0$ will be denoted as $S^{\delta \infty}$.

  \new{Note that} the exact solution solves $S^{\delta \infty} u\!=\!B' E_Y^\delta K_Y^\delta {E_Y^\delta}'\! g \!+ \!\gamma_0' u_0$. \new{From} $S^{\delta \delta}\!=\!{E_X^\delta\!}'\! S^{\delta \infty} E_X^\delta$ we \new{find} the Galerkin orthogonality $(S^{\delta \infty}(u\!-\!u^{\delta \delta}))(X^\delta)\!=\!0$.
It holds that
$$
\|\cdot\|_{X^{\delta \infty}}:=(S^{\delta \infty}\cdot)(\cdot)^{\frac12}=\sqrt{({E_Y^\delta}' B\cdot)(K_Y^\delta {E_Y^\delta}' B \cdot)+\|\gamma_0\cdot\|_H^2}
$$
\comment{formula simplified} is only a \emph{semi}-norm on $X$, which is equal to $\|\cdot\|_{X^{\delta \delta}}$ on $X^\delta$, and
\be \label{x24}
\|\cdot\|_{X^{\delta \infty}} \leq \sqrt{\kappa_{\Delta}} \|\cdot\|_X \quad \text{on } X.
\ee
\end{remark}

\section{Convergent adaptive solution method}\label{sec:adaptive}
\subsection{Preliminaries}
For $\delta \in \Delta$, we consider the modified discretized saddle point problem
(i.e.~\eqref{m9} with $({E_Y^{\hat{\delta}}}' A E_Y^{\hat{\delta}})^{-1}$ replaced by $K_Y^{\hat{\delta}}$) taking $\hat{\delta}:=\udelta$ from Assumption~\ref{saturation}.
So for a given `trial space' $X^\delta$, we employ $Y^{\udelta}$ as `test space', which is known to be sufficiently large to give stability even when employed with trial space $X^{\udelta} \supsetneq X^\delta$.
We will use this room to (adaptively) expand $X^\delta$ to some $X^{\tilde{\delta}} \subset X^{\udelta}$ while keeping $Y^{\udelta}$ fixed. Then in a second step we adapt the test space to the new trial space, i.e., replace $Y^{\udelta}$ by $Y^{\ubar{\tilde{\delta}}}$. By doing so
will construct a sequence $(\delta_i) \subseteq \Delta$ with $\delta_i \preceq \delta_{i+1}$ such that $(u^{\ubar{{\delta_i}}\delta_i})_i$ converges $r$-linearly to $u$.

As a first step, in the next lemma it is shown that if one constructs from $w \in X^\delta$ a $v \in X^\udelta$ that is closer to the best approximation $u_\udelta$ to $u$ from $X^\udelta$, then, thanks to Assumption~\ref{saturation},
$v$ is also closer to $u$.

\begin{lemma} \label{m4} Let $w \in X^\delta$, $v \in X^\udelta$ be such that for some $\rho \leq 1$,
$$
\|u_\udelta-v\|_X \leq \rho \|u_\udelta-w\|_X.
$$
Then
$$
\|u-v\|_X \leq \sqrt{\zeta^2+\rho^2(1-\zeta^2)}\,\|u-w\|_X.
$$
\end{lemma}

\begin{proof} Using $u-u_\udelta \perp_X X^\udelta$ twice, we obtain
\begin{align*}
\|u-v\|_X^2& =\|u-u_\udelta\|_X^2+\|u_\udelta-v\|_X^2\\
&\leq \|u-u_\udelta\|_X^2+\rho^2 \|u_\udelta-w\|_X^2\\
&=\|u-u_\udelta\|_X^2+\rho^2 (\|u-w\|_X^2-\|u-u_\udelta\|_X^2)\\
&=(1-\rho^2)\|u-u_\udelta\|_X^2+\rho^2\|u-w\|_X^2\\
&\leq (\zeta^2(1-\rho^2)+\rho^2)\|u-w\|_X^2,
\end{align*}
  where we used Assumption~\ref{saturation} and $\|u-u_\delta\|_X \leq \|u-w\|_X$.
\end{proof}

Notice that $u^{\udelta \delta}$ is the Galerkin approximation
from $X^\delta$ to the solution $u^{\udelta \udelta} \in X^\udelta$ of the system $S^{\udelta \udelta} u^{\udelta \udelta}=f^{\udelta \udelta}$, i.e., it is its best approximation from $X^\delta$ w.r.t.~$\|\cdot\|_{X^{\udelta \udelta}}$.
In the next proposition it is shown that an improved Galerkin approximation from an intermediate space $X^\udelta \supseteq X^{\tilde{\delta}} \supseteq X^\delta$, i.e., the function $u^{\udelta \tilde{\delta}}$, is, for $\frac{\kappa_\Delta}{\gamma_\Delta}-1$ sufficiently small, also an improved approximation to $u$, \new{and that} this holds true also for
$u^{\ubar{\tilde{\delta}} \tilde{\delta}}$. The latter function will be the successor of $u^{\udelta \delta}$ in our converging sequence.

\begin{proposition}\label{m5}
Let $\delta \preceq \tilde{\delta} \preceq \udelta$ be such that
\be \label{m7}
\|u^{\udelta \udelta}-u^{\udelta \tilde{\delta}}\|_{X^{\udelta \udelta}} \leq \rho\|u^{\udelta \udelta}-u^{\udelta \delta}\|_{X^{\udelta \udelta}}.
\ee
Then
it holds that
$$
\|u-u^{\ubar{\tilde{\delta}} \tilde{\delta}}\|_X \leq \underbrace{\frac{\kappa_\Delta}{\gamma_\Delta}  \sqrt{\zeta^2+\hat{\rho}^2(1-\zeta^2)}}_{\bar{\rho}:=}\,\|u-u^{\udelta \delta}\|_X,
$$
where $\hat{\rho}:={\textstyle \big(1+\rho \frac{\sqrt{\kappa_\Delta}}{\gamma_\Delta}\big)\sqrt{\frac{\kappa_\Delta^2}{\gamma_\Delta^2} -1}\,\sqrt{\frac{\zeta^2}{1-\zeta^2}}+\rho \frac{\sqrt{\kappa_\Delta}}{\gamma_\Delta}}$.
Notice that $\hat{\rho}$ and $\bar{\rho}$ are $<1$ when $\rho<1$ and
$\frac{\kappa_\Delta}{\gamma_\Delta}-1$ is sufficiently small dependent on $\rho$ with $\frac{\kappa_\Delta}{\gamma_\Delta}-1 \downarrow 0$ when $\rho \uparrow 1$.
\end{proposition}

\begin{proof} Using that $u-u_\udelta \perp_X X^\udelta$, it follows that
$\|u-u^{\udelta \delta}\|_X^2 \leq \frac{\kappa_\Delta^2}{\gamma_\Delta^2} \|u-u_\delta\|_X^2$ (\eqref{m6b}) is equivalent to $\|u_\udelta-u^{\udelta \delta}\|_X \leq \sqrt{\frac{\kappa_\Delta^2}{\gamma_\Delta^2} -1}\, \|u-u_{\delta}\|_X$.
Similarly, Assumption~\ref{saturation} is equivalent to $\|u-u_\delta\|_X \leq \sqrt{\frac{\zeta^2}{1-\zeta^2}} \|u_\udelta-w\|_X$ for any $w \in X^\delta$.
Additionally using \eqref{m19b}, we infer that
\begin{align*}
\|u_\udelta-u^{\udelta \tilde{\delta}}\|_X &\leq \|u_\udelta-u^{\udelta \udelta}\|_X+\|u^{\udelta \udelta}-u^{\udelta \tilde{\delta}}\|_X\\
&\leq \|u_\udelta-u^{\udelta \udelta}\|_X+\frac{\sqrt{\kappa_\Delta}}{\gamma_\Delta} \|u^{\udelta \udelta}-u^{\udelta \tilde{\delta}}\|_{X^{\udelta \udelta}}\\
&\leq \|u_\udelta-u^{\udelta \udelta}\|_X+\rho \frac{\sqrt{\kappa_\Delta}}{\gamma_\Delta}  \|u^{\udelta \udelta}-u^{\udelta \delta}\|_{X}\\
&\leq \big(1+\rho \frac{\sqrt{\kappa_\Delta}}{\gamma_\Delta}\big) \|u_\udelta-u^{\udelta \udelta}\|_X+\rho \frac{\sqrt{\kappa_\Delta}}{\gamma_\Delta}  \|u_\udelta-u^{\udelta \delta}\|_{X}\\
&\leq \Big[{\textstyle \big(1+\rho \frac{\sqrt{\kappa_\Delta}}{\gamma_\Delta}\big)\sqrt{\frac{\kappa_\Delta^2}{\gamma_\Delta^2} -1}\,\sqrt{\frac{\zeta^2}{1-\zeta^2}}+\rho \frac{\sqrt{\kappa_\Delta}}{\gamma_\Delta}} \Big]\|u_\udelta-u^{\udelta \delta}\|_{X}
\end{align*}
From Lemma~\ref{m4} we conclude that
$\|u-u^{\udelta \tilde{\delta}}\|_X \leq \sqrt{\zeta^2+\hat{\rho}^2(1-\zeta^2)} \|u-u^{\udelta \delta}\|_{X}$.
Thanks to \eqref{m6b}, it holds that
$$
\|u-u^{\ubar{\tilde{\delta}} \tilde{\delta}}\|_X \leq \frac{\kappa_\Delta}{\gamma_\Delta} \|u-u_{\tilde{\delta}}\|_X
\leq \frac{\kappa_\Delta}{\gamma_\Delta}  \|u-u^{\udelta \tilde{\delta}}\|_X,
$$
which completes the proof.
\end{proof}

\subsection{Bulk chasing and a posteriori error estimation}
To realize \eqref{m7}, i.e., to construct from the Galerkin approximation $u^{\udelta \delta}$ to $u^{\udelta \udelta}$ an improved Galerkin approximation  $u^{\udelta \tilde{\delta}}$ , we apply the concept of bulk chasing \new{(or \emph{D\"{o}rfler marking})} on a collection of a posteriori error indicators that constitute an efficient and reliable error estimator. We will apply an estimator of `hierarchical basis' type (\cite{320.7}):

Let $\Theta_\delta=\{\theta_\lambda\colon \lambda \in J_\delta\} \subseteq X^\udelta$ be such that $X^\delta +\Span \Theta_\delta=X^\udelta$ and, for some constants $0<m\leq M$, for all $\delta \in \Delta$, $z \in X^\delta$ and ${\bf c} := (c_\lambda)_{\lambda \in J_\delta} \subset \R$,
\be \label{m17}
    m^2 \|z+{\bf c}^\top \Theta_\delta\|_X^2
    \leq \|z\|_X^2+\|{\bf c}\|^2
    \leq M^2 \|z+{\bf c}^\top \Theta_\delta\|_X^2.
\ee
A suitable collection $\Theta_\delta$ will be constructed in Sect.~\ref{Smapping}.

\begin{proposition} \label{m18} Assume \eqref{m17}. Let ${\bf r}^{\,\udelta}_\delta:=(f^{\udelta \udelta}-S^{\udelta \udelta} u^{\udelta \delta})(\Theta_\delta)$, being the residual vector of $u^{\udelta \delta}$. Let $J \subseteq J_\delta$ be such that for some constant $\vartheta \in (0,1]$,
$$
\|{\bf r}^{\,\udelta}_\delta|_J\| \geq \vartheta \|{\bf r}^{\,\udelta}_\delta\|,
$$
  and, for some $\tilde{\delta} \preceq \udelta$, let $X^\delta +\Span \Theta_\delta|_J \subseteq X^{\tilde{\delta}}$. Then with $\rho:=\sqrt{1-\big(\frac{m}{M} \frac{\gamma_\Delta}{\kappa_\Delta^2} \vartheta\big)^2}$, \eqref{m7} is valid,\comment{\new{removed a verbatim copy of~\eqref{m7}}} and so, when $\frac{\kappa_\Delta}{\gamma_\Delta}-1$ is sufficiently small dependent on $\vartheta$, with $\frac{\kappa_\Delta}{\gamma_\Delta}-1 \downarrow 0$ when $ \vartheta \downarrow 0$, for some constant $\rho<\bar{\rho}<1$,
$$
\|u-u^{\ubar{\tilde{\delta}} \tilde{\delta}}\|_X \leq \bar{\rho} \|u-u^{\udelta \delta}\|_X.
$$
\end{proposition}
\begin{proof}
As a consequence of \eqref{m17} and \eqref{m19b}, we have
\[
  \frac{\gamma_\Delta}{\kappa_\Delta^2} m^2 \|z+{\bf c}^\top \Theta_\delta\|_{X^{\udelta \udelta}}^2
  \leq
  \|z\|_{X^{\udelta \udelta}}^2+ \|{\bf c}\|^2
  \leq
  \frac{\kappa_\Delta^2}{\gamma_\Delta} M^2 \|z+{\bf c}^\top \Theta_\delta\|_{X^{\udelta \udelta}}^2.
\]
We infer that
\be \label{m10}
\begin{split}
  \|u^{\udelta \tilde{\delta}}-u^{\udelta \delta}\|_{X^{\udelta \udelta}}&=\sup_{0 \neq (z,{\bf c}) \in X^\delta \times \R^{\# J_\delta}}
  \frac{(S^{\udelta \udelta}(u^{\udelta \tilde{\delta}}-u^{\udelta \delta}))(z+{\bf c}^\top \Theta_\delta)}{\|z+{\bf c}^\top \Theta_\delta\|_{X^{\udelta \udelta}}}\\
  & \geq m \frac{\sqrt{\gamma_\Delta}}{\kappa_\Delta}   \sup_{0 \neq (z,{\bf c}) \in X^\delta \times \R^{\# J_\delta}}
  \frac{(S^{\udelta \udelta}(u^{\udelta \tilde{\delta}}-u^{\udelta \delta}))({\bf c}^\top \Theta_\delta)}{\sqrt{\|z\|_{X^{\udelta \udelta}}^2+\|{\bf c}\|^2}}\\
  & \geq m  \frac{\sqrt{\gamma_\Delta}}{\kappa_\Delta}  \sup_{0 \neq {\bf c} \in \R^{\# J_\delta}}
  \frac{\langle {\bf c}|_J,(f^{\udelta \udelta} -S^{\udelta \udelta} u^{\udelta \delta})(\Theta_\delta|_J)\rangle}{\|{\bf c}|_J\|}\\
    &= m  \frac{\sqrt{\gamma_\Delta}}{\kappa_\Delta}  \|{\bf r}^{\,\udelta}_\delta|_J\|  \geq m  \frac{\sqrt{\gamma_\Delta}}{\kappa_\Delta}  \vartheta \|{\bf r}^{\,\udelta}_\delta\|\\
  & =m  \frac{\sqrt{\gamma_\Delta}}{\kappa_\Delta}  \vartheta \sup_{0 \neq (z,{\bf c}) \in X^\delta \times \R^{\# J_\delta}}
  \frac{(S^{\udelta \udelta}(u^{\udelta \udelta}-u^{\udelta \delta}))({\bf c}^{\top} \Theta_\delta)}{\sqrt{\|z\|_{X^{\udelta \udelta}}^2+\|{\bf c}\|^2}}\\
  & \geq \frac{m}{M}  \frac{\gamma_\Delta}{\kappa_\Delta^2}  \vartheta
\|u^{\udelta \udelta}-u^{\udelta \delta}\|_{X^{\udelta \udelta}}.
\end{split}
\ee
so that
\begin{align*}
\|u^{\udelta \udelta}-u^{\udelta \tilde{\delta}}\|_{X^{\udelta \udelta}}^2&=
\|u^{\udelta \udelta}-u^{\udelta \delta}\|_{X^{\udelta \udelta}}^2-
\|u^{\udelta \tilde{\delta}}-u^{\udelta \delta}\|_{X^{\udelta \udelta}}^2
\\ &\leq \Big(1-\big(\frac{m}{M}  \frac{\gamma_\Delta}{\kappa_\Delta^2}  \vartheta\big)^2\Big) \|u^{\udelta \udelta}-u^{\udelta \delta}\|_{X^{\udelta \udelta}}^2,
\end{align*}
which completes the proof of \new{\eqref{m7}}.

The final statement follows from an application of Proposition~\ref{m5}.
\end{proof}

\new{Moreover,} $\|{\bf r}^{\,\udelta}_\delta\|$ \new{is} an efficient and reliable a posteriori estimator for $\|u-u^{\udelta \delta}\|_X$:
\begin{proposition} \label{m22} Assume \eqref{m17}. Recalling that $\zeta<1$, let $\frac{\kappa_\Delta}{\gamma_\Delta}<\frac{1}{\zeta}$. Then for $\delta \in \Delta$,
$$
{\textstyle \frac{m \frac{\sqrt{\gamma_\Delta}}{\kappa_\Delta^{3/2}}}{1+\zeta \sqrt{\frac{\kappa_\Delta^2}{\gamma_\Delta^2}-1}}} \|{\bf r}^{\,\udelta}_\delta\| \leq \|u-u^{\udelta \delta}\|_X \leq {\textstyle \frac{M \frac{\kappa_\Delta^{3/2} }{\gamma_\Delta^{3/2}}}{\sqrt{1-\zeta^2}-\zeta\sqrt{\frac{\kappa_\Delta^2}{\gamma_\Delta^2}-1}}} \|{\bf r}^{\,\udelta}_\delta\|.
$$
\end{proposition}

\begin{proof} Assumption~\ref{saturation} gives $\|u-u_\udelta\|_X \leq \zeta \|u-u^{\udelta \delta}\|_X$. By $u\!-\!u_\udelta\!\perp_{X}\! X^\udelta$, \new{this} yields
\be \label{m11}
 \|u_\udelta-u^{\udelta \delta}\|_{X} \leq \|u-u^{\udelta \delta}\|_{X} \leq {\textstyle \frac{1}{\sqrt{1-\zeta^2}}} \|u_\udelta-u^{\udelta \delta}\|_{X}.
\ee
As we already have noted in the proof of Proposition~\ref{m5}, \eqref{m6b} is equivalent to $\|u_\udelta-u^{\udelta \udelta}\|_X \leq \sqrt{\frac{\kappa_\Delta^2}{\gamma_\Delta^2} -1}\, \|u-u_{\udelta}\|_X$.
Together with Assumption~\ref{saturation}, it gives
$$
\Big|\|u_\udelta-u^{\udelta \delta}\|_X-\|u^{\udelta \udelta}-u^{\udelta \delta}\|_X \Big| \leq \zeta \sqrt{\frac{\kappa_\Delta^2}{\gamma_\Delta^2}-1}\,\|u-u^{\udelta\delta}\|_X
$$
which in combination with \eqref{m11} and $\zeta \frac{\kappa_\Delta}{\gamma_\Delta}<1$ yields
$$
{\textstyle \frac{1}{1+\zeta \sqrt{\frac{\kappa_\Delta^2}{\gamma_\Delta^2}-1}}} \|u^{\udelta \udelta}-u^{\udelta \delta}\|_X \leq \|u-u^{\udelta \delta}\|_X \leq {\textstyle \frac{1}{\sqrt{1-\zeta^2}-\zeta\sqrt{\frac{\kappa_\Delta^2}{\gamma_\Delta^2}-1}}} \|u^{\udelta \udelta}-u^{\udelta \delta}\|_X.
$$
The proof is completed by \eqref{m19b} and
  $m \frac{\sqrt{\gamma_\Delta}}{\kappa_\Delta} \|{\bf r}^{\,\udelta}_\delta\| \leq \|u^{\udelta \udelta}-u^{\udelta \delta}\|_{X^{\udelta \udelta}} \leq M \frac{\kappa_\Delta}{\sqrt{\gamma_\Delta}} \|{\bf r}^{\,\udelta}_\delta\|$, where the latter inequalities \new{follow from} \eqref{m10} \new{by} reading $(J,\vartheta,\tilde{\delta})$ as
 $(J_\delta,1,\udelta)$.
 \end{proof}

Next we present an alternative a posteriori error estimator that does not rely on \eqref{m17}, that we expect to be more accurate,
and that \new{is computable} at the cost of one additional inner product.

\begin{proposition} \label{alternative} Let $\frac{\kappa_\Delta}{\gamma_\Delta}<\frac{1}{\zeta}$. \new{With notation from Remark~\ref{x14}, for $v \in X^\delta$, we find\comment{Removed unnecessary notation}
\[
  \|u - v\|_{X^{\udelta\infty}} = \sqrt{({E_Y^\udelta}'(g-Bv))(K_Y^\udelta {E_Y^\udelta}' (g-Bv))+\|u_0-\gamma_0 v\|_H^2}.
\]}
\new{It provides a computable, and efficient and reliable a posteriori error estimator. Indeed,}
\[
 \frac{\gamma_\Delta-\zeta  \kappa_\Delta}{\sqrt{\kappa_\Delta}} \|u-v\|_X \leq
  \new{\|u-v\|_{X^{\udelta \infty}}}
\leq \sqrt{\kappa_\Delta(\zeta^2+(1+\zeta\frac{\kappa_\Delta}{\gamma_\Delta})^2)} \|u-v\|_X.
\]
\end{proposition}

\begin{proof}
  \new{The equality follows directly from $(Bu,\gamma_0 u)=(g, u_0)$.}

  \new{Recall} that \new{$\|\cdot\|_{X^{\udelta \infty}} = \|\cdot\|_{X^{\udelta \udelta}}$ on $X^{\udelta}$}, and \new{that} $(S^{\udelta \infty}(u-u^{\udelta \udelta}))(X^{\udelta})=0$, \new{so that}
\be \label{x25}
\|u-w\|^2_{X^{\udelta \infty}}=\|u-u^{\udelta \udelta}\|^2_{X^{\udelta \infty}}+\|u^{\udelta \udelta}-w\|^2_{X^{\udelta \infty}}\quad (w \in X^{\udelta}),
\ee
and $\|\cdot\|_{X^{\udelta\infty}} \leq \sqrt{\kappa_\Delta} \|\cdot\|_X$ (\eqref{x24}).
  From \eqref{m6b} and \new{Assum.}~\ref{saturation}, we have for $v \in X^\delta$
\be \label{x27}
\|u-u^{\udelta \udelta}\|_X \leq \frac{\kappa_\Delta}{\gamma_\Delta}
\|u-u_\udelta\|_X \leq \zeta  \frac{\kappa_\Delta}{\gamma_\Delta}\|u-v\|_X
\ee
From \eqref{x27}, the triangle-inequality, \eqref{m19b}, and \eqref{x25} we obtain for $v \in X^\delta$
\begin{align*}
\|u-v\|_X \leq \frac{1}{1-\zeta  \frac{\kappa_\Delta}{\gamma_\Delta}}
\|u^{\udelta \udelta}-v\|_X
\leq
\frac{\sqrt{\kappa_\Delta}}{\gamma_\Delta-\zeta  \kappa_\Delta} \|u^{\udelta \udelta}-v\|_{X^{\udelta \udelta}}
\leq
\frac{\sqrt{\kappa_\Delta}}{\gamma_\Delta-\zeta  \kappa_\Delta}
\|u-v\|_{X^{\udelta \infty}}.
\end{align*}

Conversely, we have
\begin{align*}
\|u-v\|^2_{X^{\udelta \infty}} &\stackrel{\eqref{x25}}{=} \|u-u^{\udelta\udelta}\|^2_{X^{\udelta \infty}}+ \|u^{\udelta\udelta}-v\|^2_{X^{\udelta \infty}}\\
&\hspace*{-2em}\stackrel{\hspace*{2em}\eqref{x25}, \eqref{x24}}{\leq} \|u-u_\udelta\|^2_{X^{\udelta \infty}} +\kappa_\Delta \|u^{\udelta\udelta}-v\|^2_X\\
&\hspace*{-2.5em}\stackrel{\hspace*{2.5em}\eqref{x24}, \eqref{x27}}{\leq} \kappa_\Delta \|u-u_\udelta\|^2_{X} +\kappa_\Delta (1+\zeta\frac{\kappa_\Delta}{\gamma_\Delta})^2\|u-v\|^2_X\\
&\leq  \kappa_\Delta (\zeta^2+(1+\zeta\frac{\kappa_\Delta}{\gamma_\Delta})^2)\|u-v\|^2_X.
\end{align*}
  by again applying Assumption~\ref{saturation}.
%
\end{proof}

Notice that the estimator from Proposition~\ref{alternative} is exact when $\zeta=0$ and $\kappa_\Delta=1=\gamma_\Delta$, where the one from Proposition~\ref{m18}, for $v=u^{\udelta \delta}$, is exact only when additionally $m=1=M$.

\subsection{Data oscillation} \label{Soscillation}
In view of the discussion following Assumption~\ref{saturation}, notice that \emph{all} results obtained so far that depend on the `saturation constant' $\zeta$, i.e., Lemma~\ref{m4}, and Propositions~\ref{m5}, \ref{m18}, and \ref{m22}, are only valid \new{if} $(g,u_0) \in {}^{\delta\!}G \times {}^{\delta\!}U_0$.

Let us now consider the situation that the solutions and residuals in these statements refer to solutions and residuals with the true data $(g,u_0) \in Y' \times H$ being \emph{replaced} by an approximation $({}^{\delta\!}g,{}^{\delta\!}u_0) \in {}^{\delta\!}G \times {}^{\delta\!}U_0$.
In the following we denote such solutions and residuals with an \emph{additional left superscript} $\delta$, or more generally $\tilde\delta$ when a right-hand side $({}^{\tilde\delta\!}g,{}^{\tilde\delta\!}u_0) \in {}^{\tilde\delta\!}G \times {}^{\tilde\delta\!}U_0$ has been used for their computation.

\begin{proposition} \label{m21} Assume \eqref{m17}, and let $\vartheta \in (0,1]$ be a constant.
Then for $\frac{\kappa_\Delta}{\gamma_\Delta}-1$ and a constant $\hat{\omega}>0$ both being sufficiently small dependent on $\vartheta$,
with $\max(\frac{\kappa_\Delta}{\gamma_\Delta}-1,\hat{\omega}) \downarrow 0$ when $ \vartheta \downarrow 0$,
there exists a constant $\check{\rho}<1$ such that for $J \subseteq J_\delta$ with $\|{}^{\delta\!}{\bf r}^{\,\udelta}_\delta|_J\| \geq \vartheta \|{}^{\delta\!}{\bf r}^{\,\udelta}_\delta\|$, and $X^\delta+\Span \Theta_\delta|_{J} \subseteq X^{\tilde{\delta}}$, and
$$
\max\big(\|g-{}^{\delta\!}g\|_{Y'}+\|u_0-{}^{\delta\!}u_0\|_{H}, \|g-{}^{\tilde{\delta}\!}g\|_{Y'}+\|u_0-{}^{\tilde{\delta}\!}u_0\|_{H}\big)\leq \hat{\omega} \|{}^{\delta\!}{\bf r}^{\,\udelta}_\delta\|,
$$
it holds that
$$
\|{}^{\tilde{\delta}\!}u-{}^{\tilde{\delta}\!}u^{\ubar{\tilde{\delta}} \tilde{\delta}}\|_X \leq \check{\rho} \|{}^{\delta\!}u-{}^{\delta\!}u^{\udelta\delta}\|_X.
$$
\end{proposition}

\begin{proof} In \new{the introduced} notations, the statements of Propositions~\ref{m18} and \ref{m22} read
\begin{align} \nonumber
\|{}^{\delta\!}u-{}^{\delta\!}u^{\ubar{\tilde{\delta}} \tilde{\delta}}\|_X &\leq \bar{\rho} \|{}^{\delta\!}u-{}^{\delta\!}u^{\udelta\delta}\|_X,
\intertext{and} \label{55}
\|{}^{\delta\!}{\bf r}^{\,\udelta}_\delta\| &\eqsim
\|{}^{\delta\!}u-{}^{\delta\!}u^{\udelta\delta}\|_X
\end{align}
The proof is easily completed by
$$
\left.\begin{array}{@{}r@{}}
\|{}^{\tilde{\delta}\!}u-{}^{\delta\!}u\|_X\\
\|{}^{\tilde{\delta}\!}u^{\ubar{\tilde{\delta}} \tilde{\delta}}\!-\!{}^{\delta\!}u^{\ubar{\tilde{\delta}} \tilde{\delta}}\|_X
\end{array}\right\} \!\lesssim \|{}^{\tilde{\delta}\!}g\!-\!{}^{\delta\!}g\|_{Y'}\!+\!\|{}^{\tilde{\delta}\!}u_0\!-\!{}^{\delta\!}u_0\|_{H}
\leq 2 \hat{\omega} \|{}^{\delta\!}{\bf r}^{\,\udelta}_\delta\|\eqsim \hat{\omega} \|{}^{\delta\!}u\!-\!{}^{\delta\!}u^{\udelta\delta}\|_X.\!\!\!\!\qedhere
$$
\end{proof}

In view of the latter proposition, we make the following assumption.
\begin{assumption} \label{rhs}
We assume to have maps of the following types available:
  \begin{align*}
  &\Delta \rightarrow Y' \times H \colon \delta \mapsto ({}^{\delta\!}g,{}^{\delta\!}u_0) \in {}^{\delta\!}G \times {}^{\delta\!}U_0,\\
&\eta\colon \Delta \rightarrow \R \text{ such that } \|g-{}^{\delta\!}g\|_{Y'}+\|u_0-{}^{\delta\!}u_0\|_H \leq \eta(\delta), \text{ and } \eta(\tilde \delta)\leq \eta(\delta) \text{ when } \tilde{\delta} \succeq \delta, \\
&\R_{>0} \rightarrow \Delta\colon \eps \mapsto \delta(\eps) \text{ such that } \eta(\delta(\eps)) \leq \eps.
  \end{align*}
\end{assumption}
Notice that this in particular means that for any $\eps>0$ we are able to find a $\delta \in \Delta$ and $({}^{\delta\!}g,{}^{\delta\!}u_0) \in {}^{\delta\!}G \times {}^{\delta\!}U_0$ with $\|g-{}^{\delta\!}g\|_{Y'}+\|u_0-{}^{\delta\!}u_0\|_H\leq \eps$. A specification of a suitable family $({}^{\delta\!}G,{}^{\delta\!}U_0)_{\delta \in \Delta}$ will be given in Sect.~\ref{Sconcrete3}.

Given a $\delta \in \Delta$, \new{thinking} of $({}^{\delta\!}g,{}^{\delta\!}u_0)$ being a quasi-best approximation to $(g,u_0)$ from ${}^{\delta\!}G \times {}^{\delta\!}U_0$, the difference $(g,u_0)-({}^{\delta\!}g,{}^{\delta\!}u_0)$ is often referred to as \emph{data-oscillation}.

\subsection{A convergent algorithm}
In view of the statement \new{of} Proposition~\ref{m21}, in the following we will use the short-hand notations
$$
u^\delta={}^{\delta\!}u^{\udelta\delta},\quad {\bf r}^\delta={}^{\delta\!}{\bf r}^{\,\udelta}_\delta,
$$
i.e., $u^\delta$ is the solution of
\be \label{52}
\underbrace{{E^\delta_X}'(B' E^{\udelta}_Y K_Y^{\udelta} {E^{\udelta}_Y}' B+\gamma_0'\gamma_0)E^\delta_X}_{S^{\udelta \delta}=} u^\delta=\underbrace{{E^\delta_X}' (B' E^{\udelta}_Y K_Y^{\udelta} {E^{\udelta}_Y}' {}^{\delta\!}g+\gamma_0' {}^{\delta\!}u_0)}_{f^\delta={}^{\delta\!}f^{\,\udelta}_\delta:=}.
\ee
(cf.~\eqref{m9} and Sect.~\ref{Smod}), and
\be \label{53}
{\bf r}^\delta={E^{\udelta}_X}' \big[B' E^{\udelta}_Y K_Y^{\udelta} {E^{\udelta}_Y}' ( {}^{\delta\!}g-B u^\delta)+\gamma_0' ({}^{\delta\!}u_0-\gamma_0 u^\delta)\big](\Theta_\delta)\\
\ee
Instead of solving \eqref{52} exactly, we will allow it to be solved approximately with a sufficiently small relative tolerance by the application of an iterative method.
To that end, we assume to have available a
$K_X^{\delta}={K_X^{\delta}}' \in \Lis({X^{\delta}}',X^{\delta})$
for which both
\be \label{x28}
((K_X^{\delta})^{-1} w)(w) \eqsim \|w\|_X^2 \quad(w \in X^\delta)
\ee
(i.e.~$K_X^{\delta}$ is an optimal (self-adjoint and coercive) \emph{preconditioner} for $S^{\udelta \delta}$), and which can be applied at linear cost.
Besides for an efficient iterative solving of \eqref{52}, we will use this preconditioner to compute a quantity \new{$t^\delta$} that is equivalent to the $X$-norm of the algebraic error in \new{an approximation $\tilde{u}^\delta \in X^\delta$ to $u^\delta$.
The residual vector ${\bf \tilde{r}}^\delta$  corresponding to $\tilde{u}^\delta$ is defined as in \eqref{53} by replacing $u^\delta$ by $\tilde{u}^\delta$.}

\begin{algorithm} \label{54}
{\rm%
\begin{algotab}
\> \\
\> Let $\omega>0$, $\vartheta \in (0,1]$, $0 < \xi <1$ be constants, and let $\eps>0$.\\
\> $\delta:=\delta_{\rm init} \in \Delta$, $t_\delta \eqsim \|g\|_{Y'} +\|u_0\|_H$.\\
\> \texttt{do} \=\\
\>\> \texttt{do} \= compute $\tilde{u}^\delta \in X^\delta$ with $\tilde{t}_\delta:=\sqrt{(f^\delta-S^{\udelta \delta} \tilde{u}^\delta)(K_X^\delta (f^\delta-S^{\udelta \delta} \tilde{u}^\delta))} \leq \frac{t_\delta}{2}$; $t_\delta:=\tilde{t}_\delta$\\
\> \> \> \texttt{if} $e_\delta:=\|{\bf \tilde{r}}^\delta\|+\eta(\delta)+t_\delta \leq\eps$ \texttt{then} \texttt{stop} \texttt{endif}\\
\>\> \texttt{until} $t_\delta \leq \xi e_\delta$\\
\>\> \texttt{if} $\eta(\delta)>\omega \|{\bf \tilde{r}}^\delta\|$\quad\% \new{i.e.~data-oscillation dominates approximation error.}\\
\>\>\texttt{then} select $\tilde{\delta} \in \Delta$ s.t. $X^{\tilde \delta}\supseteq X^\delta$ is (a near-smallest) space such that $\eta(\tilde{\delta})\leq \eta(\delta)/2$.\\
\>\>\texttt{else} \= determine $\delta \preceq \tilde{\delta} \preceq \udelta$ s.t.~$X^{\tilde \delta}$ is (a near- smallest) space that for a $J \subseteq I_\delta^{\udelta}$\\
\>\>\> contains $X^\delta+\Span \Theta_\delta|_J$ where  $\|{\bf \tilde{r}}^\delta|_{J}\|\geq \vartheta \|{\bf \tilde{r}}^\delta\|$.\\
\>\>\texttt{endif} \\
\>\> $t_{\tilde{\delta}}:=e_\delta$, $\delta:=\tilde{\delta}$\\
\>\texttt{enddo}
\end{algotab}}
\end{algorithm}

\begin{theorem} \label{thm:main}
Assume \eqref{m17}, and let the constants $\gamma_\Delta$ and $\kappa_{\Delta}$ be as defined in \eqref{x21} and \eqref{kappadelta}, respectively.
  For constants $\vartheta$, $\omega/\vartheta$, $\xi/\vartheta$, $(\frac{\kappa_\Delta}{\gamma_\Delta}-1)/\omega$ \new{sufficiently small}, with additionally $\omega$ and $\frac{\kappa_\Delta}{\gamma_\Delta}-1$ sufficiently small dependent on $\vartheta$ with $\max(\frac{\kappa_\Delta}{\gamma_\Delta}-1,\omega) \downarrow 0$ when $\vartheta \downarrow 0$,
there exists a constant $\breve{\rho}<1$ such that
between any two successive passings of the until-clause the value of $\vartheta \|u-\tilde{u}^\delta\|_X+\eta(\delta)$ decreases with at least a factor $\breve{\rho}$.
For any $\eps>0$, Algorithm~\ref{54} terminates, and at termination it holds that $\|u-\tilde{u}^\delta\|_X+\eta(\delta) \lesssim \eps$.\comment{Remark deleted.}
\end{theorem}

\begin{proof}
  \new{Replacing} $w$ by $K_X^\delta S^{\udelta \delta} w$ in~\eqref{x28} \new{yields} $(S^{\udelta \delta} v)(\!K_X^\delta S^{\udelta \delta} v) \!\eqsim\! \|v\|_{\!X}^2$ for $v \!\in\! X^\delta$, \new{so}
\be \label{x7}
\|u^\delta-\tilde{u}^\delta\|_X^2 \eqsim (f^\delta-S^{\udelta \delta} \tilde{u}^\delta)(K_X^\delta (f^\delta-S^{\udelta \delta} \tilde{u}^\delta)).
\ee

From \eqref{m19b} and \eqref{m17}, we have
\be \label{eqn:residual-alg}
\|{\bf r}^\delta\!-\!{\bf \tilde{r}}^\delta\|=\hspace*{-0.8em}\sup_{0 \neq {\bf c} \in \R^{\# J^\delta}} \hspace*{-0.8em}\frac{(S^{\udelta \udelta}(\tilde{u}^\delta-u^\delta))({\bf c}^\top \Theta_\delta)}{\|{\bf c}\|} \leq \tfrac{\kappa_\Delta^{\frac12}}{m} \|u^\delta-\tilde{u}^\delta\|_{X^{\udelta \udelta}} \leq \tfrac{\kappa_\Delta}{m} \|u^\delta-\tilde{u}^\delta\|_X.
\ee

If the algorithm stops, then
\begin{align*}
\|u-\tilde{u}^\delta\|_{X} &\leq \|u-{}^{\delta\!}u\|_{X} +\|{}^{\delta\!}u-u^\delta\|_X+\|u^\delta-\tilde{u}^\delta\|_X\\
&\lesssim
  \eta(\delta) +\|{}^{\delta\!}u-u^\delta\|_X+\|u^\delta-\tilde{u}^\delta\|_X \quad\text{(by Thm.~\ref{thm0} \& Ass.~\ref{rhs})}\\
  &\stackrel{\makebox[0pt]{$\scriptstyle \eqref{55}$}}{\eqsim}
 \eta(\delta)+\|{\bf r}^\delta\|+\|u^\delta-\tilde{u}^\delta\|_X \leq \eta(\delta)+\|{\bf \tilde{r}}^\delta\|+\|{\bf r}^\delta-{\bf \tilde{r}}^\delta\|+\|u^\delta-\tilde{u}^\delta\|_X \\
&\stackrel{\makebox[0pt]{$\scriptstyle \eqref{eqn:residual-alg}$}}{\lesssim}  \eta(\delta)+\|{\bf \tilde{r}}^\delta\|+\|u^\delta-\tilde{u}^\delta\|_X
\stackrel{\makebox[0pt]{$\scriptstyle \eqref{x7}$}}{\lesssim}
 \eta(\delta)+\|{\bf \tilde{r}}^\delta\|+t_\delta \leq \eps.
\end{align*}

The inner do-loop always terminates either by passing the until-clause or by the stop-statement. Indeed, inside this loop
the value of $t_\delta$ is driven to $0$, so that $\|{\bf \tilde{r}}^\delta\|+\eta(\delta)$ tends to $\|{\bf r}^\delta\|+\eta(\delta)$.
So if $\|{\bf r}^\delta\|+\eta(\delta) \neq 0$, then at some moment $t_\delta \leq \xi(\|{\bf \tilde{r}}^\delta\|+\eta(\delta) +t_\delta)$, whereas if $\|{\bf r}^\delta\|+\eta(\delta) =0$ then at some moment $e_\delta  \leq \eps$.

When passing the until-clause, it holds that $t_\delta \leq \xi(\|{\bf \tilde{r}}^\delta \|+\eta(\delta)+t_\delta)$, and so by using $\xi<1$ kicking back $t_\delta$,
\begin{align} \label{x1} 
t_\delta &\lesssim \xi(\|{\bf \tilde{r}}^\delta \|+\eta(\delta))\\ \nonumber
&\leq \xi(\|{\bf r}^\delta\|+\|{\bf \tilde{r}}^\delta-{\bf r}^\delta\|+\eta(\delta))\\\nonumber
&\lesssim \xi(\|{}^{\delta\!}u -u^\delta\|_X+\|\tilde{u}^\delta-u^\delta\|_X+\eta(\delta))\\\nonumber
&\lesssim \xi(\|u -u^\delta\|_X+t_\delta+\eta(\delta))
\end{align}
Taking $\xi$ small enough and kicking back $t_\delta$, we obtain
$t_\delta \lesssim \xi (\|u -u^\delta\|_X+\eta(\delta)),
$
and similarly
\begin{align} \label{x3} 
t_\delta &\lesssim \xi (\|u -\tilde{u}^\delta\|_X+\eta(\delta)),\\ \label{x13}
t_\delta &\lesssim \xi (\|{}^{\delta\!}u -u^\delta\|_X+\eta(\delta)).
\end{align}
When passing the until-clause, furthermore we have
\be \label{x2} 
\begin{split}
\hspace*{-2em}\|u-\tilde{u}^\delta\|_X &\lesssim t_\delta+\|u-u^\delta\|_X \lesssim  t_\delta+\|{}^{\delta\!}u-u^\delta\|_X+\eta(\delta) \eqsim
t_\delta+\|{\bf r}^\delta\|+\eta(\delta)\\
& \leq
t_\delta+\|{\bf r}^\delta-{\bf \tilde{r}}^\delta\|+\|{\bf \tilde{r}}^\delta\|+\eta(\delta)
\lesssim t_\delta+\|u^\delta-\tilde{u}^\delta\|_X+\|{\bf \tilde{r}}^\delta\|+\eta(\delta)\\
&\lesssim t_\delta+\|{\bf \tilde{r}}^\delta\|+\eta(\delta)
\stackrel{\eqref{x1}}{\lesssim} \|{\bf \tilde{r}}^\delta\|+\eta(\delta).
\end{split}
\ee

Denoting $\delta$ at the subsequent passing of the until-clause as $\tilde{\delta}$, we have
\begin{align*}
\hspace*{-2em}\|u-\tilde{u}^{\tilde \delta}\|_X & \lesssim t_{\tilde \delta}+\|{}^{\tilde \delta\!}u -u^{\tilde \delta}\|_X+\eta(\tilde \delta)\\
 & \stackrel{\makebox[0pt]{$\scriptstyle \eqref{m6b}$}}{\leq} t_{\tilde \delta}+\frac{\kappa_\Delta}{\gamma_\Delta} \|{}^{\tilde \delta\!}u -{}^{\tilde \delta\!}u_{\tilde \delta}\|_X+\eta(\tilde \delta)\\
 & \stackrel{\makebox[0pt]{$\scriptstyle \delta \preceq \tilde{\delta}$}}{\leq} t_{\tilde \delta}+\frac{\kappa_\Delta}{\gamma_\Delta} \|{}^{\tilde \delta\!}u -\tilde{u}^\delta\|_X+\eta(\tilde \delta)\\
  & \leq t_{\tilde \delta}+\frac{\kappa_\Delta}{\gamma_\Delta} \|u -\tilde{u}^\delta\|_X+{\mathcal O}(\eta(\tilde \delta)).
\end{align*}
By using $t_{\tilde \delta} \lesssim \xi (\|u -\tilde{u}^{\tilde \delta}\|_X+\eta(\tilde \delta))$ (\eqref{x3}) and kicking back $\|u -\tilde{u}^{\tilde \delta}\|_X$, we infer that for $\xi$ sufficiently small,
\be \label{x5}
\|u-\tilde{u}^{\tilde \delta}\|_X \leq (\frac{\kappa_\Delta}{\gamma_\Delta}+{\mathcal O}(\xi)) \|u -\tilde{u}^\delta\|_X +{\mathcal O}(\eta(\tilde \delta)).
\ee

In the case that
\be \label{x4}
\eta(\delta) > \omega\|{\bf \tilde{r}}^\delta\|,
\ee
it holds that
$$
\eta(\tilde \delta) \leq \eta(\delta)/2,
$$
and so, thanks to \eqref{x2} and \eqref{x4},
\be \label{x15}
\|u-\tilde{u}^\delta\|_X \lesssim \eta(\delta)/\omega.
\ee
For any constant $\rho_1<1$, using \eqref{x5} and \eqref{x15} we have
\begin{align*}
\|u-\tilde{u}^{\tilde \delta}\|_X &\leq \rho_1\|u -\tilde{u}^\delta\|_X +{\textstyle \frac{\kappa_\Delta / \gamma_\Delta+{\mathcal O}(\xi) - \rho_1}{\omega}} \omega \|u -\tilde{u}^\delta\|_X+{\mathcal O}(\eta(\tilde \delta))\\
&\leq \rho_1\|u -\tilde{u}^\delta\|_X +\big({\textstyle \frac{\kappa_\Delta / \gamma_\Delta+{\mathcal O}(\xi) - \rho_1}{\omega}}+1\big)C \eta(\delta),
\end{align*}
for some constant $C>0$,
and so
$$
\vartheta \|u-\tilde{u}^{\tilde \delta}\|_X +\eta(\tilde \delta)
\leq \rho_1 \vartheta \|u -\tilde{u}^\delta\|_X +\Big(\vartheta \big({\textstyle \frac{\kappa_\Delta / \gamma_\Delta-\rho_1+{\mathcal O}(\xi)}{\omega}}+1\big)C +{\textstyle \frac12}\Big)\eta(\delta)
$$
Now let $\vartheta>0$ be such that $2 \vartheta C+\frac12<1$.
Given a constant $\omega$ (which later will be selected such that $\omega/\vartheta$ is sufficiently small),
let $(\frac{\kappa_\Delta}{\gamma_\Delta}-1)/\omega$, $(1-\rho_1)/\omega$, $\xi/\omega$ be sufficiently small such that the expression $\frac{\kappa_\Delta / \gamma_\Delta-\rho_1+{\mathcal O}(\xi)}{\omega}\leq 1$.
We conclude that in this case $\vartheta \|u-\tilde{u}^{\delta}\|_X +\eta(\delta)$ is reduced by at least a factor $\max(\rho_1,2\vartheta C+\frac12)<1$.
\bigskip

Next we consider the other case that
\be \label{x88}
\eta(\delta) \leq \omega\|{\bf \tilde{r}}^\delta\|,
\ee
so that
\be \label{x6}
\|{\bf \tilde{r}}^\delta|_{J}\| \geq \vartheta \|{\bf \tilde{r}}^\delta\|.
\ee
We have
\begin{align} \label{x11} 
\|{\bf r}^\delta-{\bf \tilde{r}}^\delta\| & \lesssim \|u^\delta-\tilde{u}^\delta\|_X \lesssim t_\delta
\stackrel{\makebox[0pt]{$\scriptstyle  \eqref{x1}$}}{\leq} \xi (\|{\bf \tilde{r}}^\delta\|+\eta(\delta))
\stackrel{\makebox[0pt]{$\scriptstyle  \eqref{x88}$}}{\lesssim} \xi \|{\bf \tilde{r}}^\delta\| \\ \nonumber
& \leq  \xi (\|{\bf r}^\delta\|+\|{\bf r}^\delta-{\bf \tilde{r}}^\delta\|),
\end{align}
and so by taking $\xi$ sufficiently small and kicking back $\|{\bf r}^\delta-{\bf \tilde{r}}^\delta\|$, also
\be
 \label{x9} 
\|{\bf r}^\delta-{\bf \tilde{r}}^\delta\| \lesssim  \xi \|{\bf r}^\delta\|,
\ee
which together with \eqref{x88} implies that
\be \label{x8}
\eta(\delta) \lesssim \omega\|{\bf r}^\delta\|.
\ee
From \eqref{x6}--\eqref{x9} we infer that
$$
\|{\bf r}^\delta\| \stackrel{\makebox[0pt]{$\scriptstyle  \eqref{x11}$}}{\lesssim} \|{\bf \tilde{r}}^\delta\| \stackrel{\makebox[0pt]{$\scriptstyle  \eqref{x6}$}}{\lesssim}
\|{\bf \tilde{r}}^\delta|_{J}\| \leq \|{\bf r}^\delta|_{J}\|+\|{\bf r}^\delta-{\bf \tilde{r}}^\delta\|
\stackrel{\makebox[0pt]{$\scriptstyle  \eqref{x9}$}}{\lesssim}
\|{\bf r}^\delta|_{J}\|+\xi \|{\bf r}^\delta\|.
$$
By taking $\xi$ sufficiently small and kicking back $\|{\bf r}^\delta\|$, we conclude that
there exists a constant $\tilde{\vartheta}>0$ such that
$$
\|{\bf r}^\delta|_{J}\| \geq \tilde{\vartheta} \|{\bf r}^\delta\|.
$$

Assuming that $\frac{\kappa_\Delta}{\gamma_\Delta}-1$ and $\omega$ are small enough, using \eqref{x8}
an application of Proposition~\ref{m21} shows that there exists a constant
$\rho_3<1$ such that
\be \label{x17}
\|{}^{\tilde \delta\!}u - u^{\tilde \delta}\|_X \leq \rho_3 \|{}^{\delta\!}u - u^{\delta}\|_X.
\ee

Furthermore we have
\begin{align*}
\|{\bf \tilde{r}}^\delta\| & \leq \|{\bf r}^\delta\| +\|{\bf \tilde{r}}^\delta-{\bf r}^\delta\| \lesssim \|{}^{\delta\!}u - u^{\delta}\|_X +t_\delta \lesssim \|u-\tilde{u}^\delta\|_X+\eta(\delta)+t_\delta\\
&\stackrel{\makebox[0pt]{$\scriptstyle  \eqref{x3}$}}{\lesssim} \|u-\tilde{u}^\delta\|_X+\eta(\delta),
\end{align*}
and so from \eqref{x88} by kicking back $\eta(\delta)$ and taking $\omega$ sufficiently small,
\be \label{x10}
\eta(\delta) \lesssim \omega\|u-\tilde{u}^\delta\|_X.
\ee
We conclude that
\begin{align*}
\vartheta \|u-\tilde{u}^{\tilde{\delta}}\|_X+\eta(\tilde \delta) &\leq \vartheta \|{}^{\tilde \delta\!}u-u^{\tilde{\delta}}\|_X+{\mathcal O}(\eta(\tilde \delta)+t_{\tilde \delta})\\
&\stackrel{\makebox[0pt]{$\scriptstyle  \eqref{x13}$}}{\leq}
(\vartheta +{\mathcal O}(\xi)) \|{}^{\tilde \delta\!}u-u^{\tilde{\delta}}\|_X +{\mathcal O}(\eta(\tilde \delta))\\
&\stackrel{\makebox[0pt]{$\scriptstyle  \eqref{x17}$}}{\leq} (\vartheta +{\mathcal O}(\xi)) \rho_3  \|{}^{\delta\!}u-u^{\delta}\|_X+{\mathcal O}(\eta(\tilde \delta))\\
&\leq (\vartheta +{\mathcal O}(\xi)) \rho_3  \big[\|u-\tilde{u}^{\delta}\|_X+{\mathcal O}(\eta(\delta))+t_\delta\big]+{\mathcal O}(\eta(\tilde \delta))\\
& \stackrel{\makebox[0pt]{$\scriptstyle  \eqref{x3}$}}{\leq} (\vartheta +{\mathcal O}(\xi)) \rho_3  \big[(1+{\mathcal O}(\xi))\|u-\tilde{u}^{\delta}\|_X+{\mathcal O}(\eta(\delta))\big]+{\mathcal O}(\eta(\delta))\\
& \stackrel{\makebox[0pt]{$\scriptstyle  \eqref{x10}$}}{\leq}\big[(\vartheta +{\mathcal O}(\xi)) \rho_3 (1+{\mathcal O}(\xi))+{\mathcal O}(\omega)\big]\|u-\tilde{u}^{\delta}\|_X\\
& = \big[\rho_3 +{\mathcal O}({\textstyle \frac{\omega+\xi}{\vartheta}})\big] \vartheta \|u-\tilde{u}^{\delta}\|_X.
\end{align*}
So for $\omega/\vartheta$ and $\xi/\vartheta$ sufficiently small, also in this case we established  a reduction of $\vartheta \|u-\tilde{u}^{\tilde{\delta}}\|_X+\eta(\tilde \delta)$
by at least a constant factor less than $1$.
\bigskip

\new{It} is left to show that the algorithm terminates.
We \new{showed} that the value of $\vartheta \|u-\tilde{u}^\delta\|_X+\eta(\delta)$ at passing the until-clause is $r$-linearly converging.
We consider the corresponding value of $e_\delta=\|{\bf \tilde{r}}^\delta\|+\eta(\delta)+t_\delta$. \new{Arguments used} multiple times show that
$\|{\bf \tilde{r}}^\delta\| \lesssim \|u-\tilde{u}^\delta\|_X+\eta(\delta)+t_\delta$, and so
$e_\delta \lesssim \|u-\tilde{u}^\delta\|_X+\eta(\delta)+t_\delta$.
Using that $t_\delta \leq \xi(\|{\bf \tilde{r}}^\delta\|+\eta(\delta))\leq \xi e_\delta$, for $\xi$ sufficiently small kicking back $e_\delta$ shows that
$$
e_\delta \lesssim \|u-\tilde{u}^\delta\|_X+ \eta(\delta),
$$
and so
$$
e_\delta \lesssim \vartheta \|u-\tilde{u}^\delta\|_X+ \eta(\delta).
$$
\new{Therefore}, at some moment $e_\delta \leq \eps$, meaning that the algorithm stops.
\end{proof}



\section{\new{Wavelets-in-time} tensorized with finite-elements-in-space} \label{Srealization}
We specify the parabolic problem at hand, as well as the type of families $(X^\delta)_{\delta \in \Delta}$ and $(Y^\delta)_{\delta \in \Delta}$ of `trial' and `test' spaces.
\new{The specification of} a likely harmless minor further restriction to these families that will be needed for the construction of an $X$-stable collection $\Theta_\delta$ that spans an $X$-stable complement space of $X^\delta$ in $X^{\udelta}$, specifically condition \eqref{m17}, will be postponed to Sect.~\ref{Smapping}.

\subsection{Continuous problem}
For some bounded domain $\Omega \subset \R^d$, we take $H=L_2(\Omega)$ and, for some closed \new{$\Gamma_D \subseteq \partial\Omega$}, $V=H^1_{0,\Gamma_D}(\Omega):=\clos_{H^1(\Omega)}\{u \in C^\infty(\Omega) \cap H^1(\Omega)\colon u|_{\Gamma_D}=0\}$,
and
$$
a(t;\eta,\zeta):=\int_\Omega {\bf K} \nabla \eta \cdot \nabla \zeta +c \eta \zeta \,d{\bf x},
$$
where ${\bf K}={\bf K}^\top \in L_\infty(I \times \Omega)$ with ${\bf K}(\cdot) \eqsim \identity$ a.e., $c \in L_\infty(I \times \Omega)$ \new{with $\essinf c \geq 0$, and $|\Gamma_D|>0$ when the latter $\essinf$ is zero}.
W.l.o.g.~we take $T=1$.

\subsection{Wavelets in time} \label{Swavsintime}
We will construct the trial and test spaces as the span of wavelets-in-time tensorized with finite element spaces-in-space. In this subsection we collect some assumptions on the wavelets.

At the \emph{`trial side'} we consider a countable collection $\Sigma=\{\sigma_\lambda\colon \lambda \in \vee_\Sigma\}$ of functions $I \rightarrow \R$ known as \emph{wavelets}.
To each $\lambda \in \vee_\Sigma$ we associate a value $|\lambda| \in \N_0$, called the \emph{level} of $\lambda$.
We assume that the wavelets are \emph{locally supported} meaning that
$\sup_{n \in \N,\,\ell \in \N_0} \#\{\lambda \in \vee_\Sigma \colon |\lambda|=\ell,\,|\supp \sigma_\lambda \cap 2^{-\ell}(n+[0,1])|>0\} <\infty$ and $\diam \supp \sigma_\lambda \lesssim 2^{-|\lambda|}$.
To each $\lambda \in \vee_\Sigma$ with $|\lambda| >0$, we associate one or more $\tilde{\lambda} \in \vee_\Sigma$ with $|\tilde{\lambda}|=|\lambda|-1$ and $|\supp \sigma_\lambda \cap \supp \sigma_{\tilde{\lambda}}|>0$ which we call the \emph{parent(s)} of $\lambda$.
We denote this relation between a parent $\tilde{\lambda}$ and a child $\lambda$ by
$$
\tilde{\lambda} \triangleleft_\Sigma \lambda.
$$
The definitions of parents and children give rise to obvious notions of ancestors and descendants.

To each $\lambda \in \vee_\Sigma$ we associate some neighbourhood $\cS_\Sigma(\lambda)$ of $\supp \sigma_\lambda$ with $\diam\cS_\Sigma(\lambda) \lesssim 2^{-|\lambda|}$ and
$$
\tilde{\lambda} \triangleleft_\Sigma \lambda \Longrightarrow \cS_\Sigma(\tilde{\lambda}) \supseteq\cS_\Sigma(\lambda).
$$
For some wavelets bases, e.g.~Alpert wavelets (\cite{10}), it suffices to take $\cS_\Sigma(\lambda)=\supp \sigma_\lambda$. With $C_\Sigma:=\sup_{\lambda \in \vee_\Sigma} 2^{|\lambda|} \diam \supp \sigma_\lambda$,
a neighbourhood that in any case is sufficiently large is
$\{t \in I \colon \dist(t,\supp \sigma_\lambda) \leq C_\Sigma 2^{-|\lambda|}\}$.
Indeed, if with this definition $t \in \cS_\Sigma(\lambda)$ and $\tilde{\lambda} \triangleleft_\Sigma \lambda$, then $\dist(t,\supp \sigma_{\tilde{\lambda}}) \leq \dist(t,\supp \sigma_{\lambda})+\diam(\supp \sigma_{\lambda}) \leq 2 C_\Sigma 2^{-|\lambda|}$, i.e.~$t \in \cS_\Sigma(\tilde{\lambda})$.

We assume that $\Sigma$ is a Riesz basis for $L_2(I)$, and, when renormalized in $H^1(I)$-norm, it is a Riesz basis for $H^1(I)$.
Although not essential, thinking of wavelets being (essentially) constructed by means of \emph{dilation}, we assume that
$$
\|\sigma_\lambda\|_{H^1(I)} \eqsim 2^{|\lambda|}.
 $$

At the \emph{`test side'} we consider a similar collection $\Psi=\{\psi_\mu\colon \mu \in \vee_\Psi\}$ of wavelets, with the difference though that this one has to be an even \emph{orthonormal} basis for $L_2(\Omega)$, \new{whilst these wavelets do not need to be in $H^1(I)$.}

We will assume that $\Sigma$ and $\Psi$ are selected such that for any $\ell \in N_0$,
$$
\Span\{\sigma_\lambda\colon |\lambda|\leq \ell \} \cup \Span\{\sigma'_\lambda\colon |\lambda|\leq \ell\} \subseteq \Span\{\psi_\mu\colon |\mu|\leq \ell\},
$$
so that in particular
\be \label{x16}
|\mu|>|\lambda| \Longrightarrow \langle \sigma_\lambda, \psi_\mu\rangle_{L_2(I)}=0=\langle \sigma'_\lambda, \psi_\mu\rangle_{L_2(I)}.
\ee

\subsection{Uniform stability}
In the following proposition, we further specify the type of families of trial and test spaces that we consider, and formulate sufficient conditions for the requirements \eqref{x20}--\eqref{x21}, which implied
uniform stability of the Galerkin discretizations of our saddle-point problem \eqref{m1}.

\begin{proposition} \label{prop1} For $\delta \in \Delta$, let
$$
X^\delta=\sum_{\lambda \in \vee_\Sigma} \sigma_\lambda \otimes W_\lambda^\delta,\quad
Y^\delta=\sum_{\mu \in \vee_\Psi} \psi_\mu \otimes V_\mu^\delta
$$
for subspaces $W_\lambda^\delta, V_\mu^\delta \subseteq V$ of which finitely many are non-zero.
\new{Assume}
\begin{align} \label{x18}
& \langle \sigma_\lambda,\psi_\mu \rangle_{L_2(I)} \neq 0 \Longrightarrow
V_\mu^\delta \supseteq W_\lambda^\delta,
\intertext{and, for some constant $\gamma_\Delta>0$, for any $\mu \in \vee_\Psi$,} \label{x19}
& \inf_{0 \neq w \in \sum\limits_{\{\lambda \in \vee_\Sigma\colon \langle \sigma_\lambda',\psi_\mu \rangle_{L_2(I)} \neq 0\}}W_\lambda^\delta}  \,\,\sup_{0 \neq v \in V_\mu^\delta} \frac{w(v)}{\|w\|_{V'} \|v\|_{V}} \geq \gamma_\Delta.
\end{align}
Then $X^\delta \subseteq Y^\delta$ and
$$
\inf_{0 \neq w \in X^\delta} \sup_{0 \neq v \in Y^\delta} \frac{(\partial_t w)(v)}{\|\partial_t w\|_{Y'} \|v\|_{Y}} \geq \gamma_\Delta,
$$
i.e., the conditions \eqref{x20}--\eqref{x21} for uniform stability are satisfied.
\end{proposition}

\begin{proof}
For $w_\lambda \in W_\lambda^\delta$ and $w:=\sum_{\lambda \in \vee_\Sigma} \sigma_\lambda \otimes w_\lambda \in X^\delta$,
$\sigma_\lambda =\sum_{\mu \in \vee_\Psi} \langle \sigma_\lambda,\psi_\mu\rangle_{L_2(I)} \psi_\mu$ shows that
$w=\sum_{\mu \in \vee_\Psi} \psi_\mu \otimes \sum_{\lambda\in \vee_\Sigma} \langle \sigma_\lambda,\psi_\mu\rangle_{L_2(I)} w_\lambda \in Y^\delta$ by the first assumption.
Similarly
$\partial_t w=\sum_{\mu \in \vee_\Psi} \psi_\mu \otimes \tilde{v}_\mu$ where $\tilde{v}_\mu:=\sum_{\lambda\in \vee_\Sigma} \langle \sigma'_\lambda,\psi_\mu\rangle_{L_2(I)} w_\lambda$.
For any $\eps>0$, the second assumption shows that for any $\mu \in \vee_\Psi$, there exists a $v_\mu \in V_\mu^\delta$ with $ \tilde{v}_\mu(v_\mu) \geq (\gamma_\Delta-\eps) \|\tilde{v}_\mu\|_{V'} \|v_\mu\|_V$ and $\|\tilde{v}_\mu\|_{V'}=\|v_\mu\|_V$. With $v:=\sum_{\mu \in \vee_\Psi}\psi_\mu \otimes v_\mu \in Y^\delta$, we infer that
$(\partial_t w)(v)=\sum_{\mu \in \vee_\Psi}  \tilde{v}_\mu(v_\mu)  \geq (\gamma_\Delta-\eps) \|\partial_t w\|_{Y'} \|v\|_{Y}$.
\end{proof}


In order to be able to apply at linear cost
the arising linear operators in \eqref{52}--\eqref{53}, we will restrict the type of trial spaces
 $X^\delta=\sum_{\lambda \in \vee_\Sigma} \sigma_\lambda \otimes W_\lambda^\delta$ by imposing the following \emph{tree condition}
\be \label{tree-in-time1}
\tilde{\lambda} \triangleleft_\Sigma \lambda \Longrightarrow W^\delta_{\tilde{\lambda}} \supseteq W^\delta_\lambda.
\ee
For the same reason the analogous condition will be needed for $Y^\delta$.
For $X^\delta$ that satisfies \eqref{tree-in-time1}, below the latter will be verified, and sufficient, more easily verifiable conditions for \eqref{x18}--\eqref{x19} are derived.


\begin{proposition} \label{prop2}
Let $X^\delta=\sum_{\lambda \in \vee_\Sigma} \sigma_\lambda \otimes W_\lambda^\delta$ satisfy \eqref{tree-in-time1}.
For $\mu \in \vee_\Psi$, set
$$
\invbreve{W}_\mu^\delta:=\sum\limits_{\{\lambda \in \vee_\Sigma\colon |\lambda|=|\mu|,\,|\cS_\Psi(\mu) \cap \cS_\Sigma(\lambda)|>0\}} W_\lambda^\delta.
$$
Build $Y^\delta=\sum_{\mu \in \vee_\Psi} \psi_\mu \otimes V_\mu^\delta$ by taking $V_\mu^\delta=\{0\}$ when $\invbreve{W}_\mu^\delta=\{0\}$, and otherwise
\begin{align} \nonumber
&V_\mu^\delta \supseteq \invbreve{W}_\mu^\delta
\intertext{where} \label{x22b}
\inf_{0 \neq w \in \invbreve{W}_\mu^\delta}  \sup_{0 \neq v \in V_\mu^\delta} & \frac{w(v)}{\|w\|_{V'} \|v\|_{V}}  \geq \gamma_\Delta
\end{align}
for some constant $\gamma_\Delta>0$.
Then the conditions \eqref{x18} and \eqref{x19} from Proposition~\ref{prop1} for uniform stability are satisfied.

When $\dim V_\mu^\delta \lesssim \dim \invbreve{W}_\mu^\delta$, then $\dim Y^\delta \lesssim \dim X^\delta$, and
under the natural condition that a larger $\invbreve{W}_\mu^\delta$ gives rise to a larger (more precisely, not smaller) $V_\mu^\delta$, the constructed $Y^\delta$ satisfies the \emph{tree condition}
\be \label{tree-in-time2}
\tilde{\mu} \triangleleft_\Psi \mu \Longrightarrow V_{\tilde{\mu}} \supseteq V_\mu.
\ee
\end{proposition}

\begin{proof} Let $\langle \sigma_\lambda,\psi_\mu\rangle_{L_2(I)}\neq 0$ or $\langle \sigma_\lambda',\psi_\mu\rangle_{L_2(I)}\neq 0$. Then $|\cS_\Sigma(\lambda) \cap \cS_\Psi(\mu)|>0$ and $|\lambda| \geq |\mu|$ by \eqref{x16}.
When $|\lambda|>|\mu|$, $\lambda$ has an ancestor $\tilde{\lambda}$ with $|\tilde{\lambda}|=|\mu|$, $W_{\tilde{\lambda}}^\delta \supseteq W_\lambda^\delta$, and $\cS_\Sigma(\tilde{\lambda}) \supseteq \cS_\Sigma(\lambda)$, and thus $|\cS_\Sigma(\tilde{\lambda}) \cap \cS_\Psi(\mu)|>0$.
We conclude that both
$\sum_{\{\lambda \in \vee_\Sigma\colon \langle \sigma_\lambda,\psi_\mu\rangle_{L_2(I)}\neq 0\}} W_\lambda^\delta$
and
$\sum_{\{\lambda \in \vee_\Sigma\colon \langle \sigma_\lambda',\psi_\mu\rangle_{L_2(I)}\neq 0\}} W_\lambda^\delta$
are included in
$\invbreve{W}_\mu^\delta$, so that  \eqref{x18} and \eqref{x19}  are guaranteed by the selection of $V_\mu^\delta$.

The statement $\dim Y^\delta \lesssim \dim X^\delta$ when $\dim V_\mu^\delta \lesssim \dim \invbreve{W}_\mu^\delta$ follows from
$\dim \invbreve{W}_\mu^\delta \leq \sum_{\{\lambda \in \vee_\Sigma\colon |\lambda|=|\mu|,\,|\cS_\Psi(\mu) \cap \cS_\Sigma(\lambda)|>0\}} \dim W_\lambda^\delta$,
and the fact that for any $\lambda \in \vee_\Sigma$, the number of $\mu \in \vee_\Psi$ with $|\mu|=|\lambda|$ and $|\cS_\Psi(\mu) \cap \cS_\Sigma(\lambda)|>0$ is uniformly bounded.

Let $\tilde{\mu} \triangleleft_\Psi \mu$, and so $\cS_\Psi(\tilde{\mu}) \supseteq \cS_\Psi(\mu)$.
For each $\lambda \in \vee_\Sigma$ with $|\lambda|=|\mu|$ and $|\cS_\Psi(\mu) \cap \cS_\Sigma(\lambda)|>0$, there exists a $\tilde{\lambda} \triangleleft_\Sigma \lambda$, thus with $\cS_\Sigma(\tilde{\lambda}) \supseteq \cS_\Sigma(\lambda)$, and $W_{\tilde{\lambda}}^\delta \supseteq W_{\lambda}^\delta$ by \eqref{tree-in-time1}.
We conclude that $\invbreve{W}_{\tilde{\mu}}^\delta \supseteq \invbreve{W}_\mu^\delta$, which completes the proof of \eqref{tree-in-time2}.
\end{proof}

As follows from \cite[{\new{Thm.~3.11}}]{58.6} (taking $B$ to be the Riesz map $H \rightarrow H'$) condition \eqref{x22b} has the following equivalent formulation.
\begin{proposition}\label{prop3} Condition \eqref{x22b} is equivalent to existence of \new{uniformly bounded projectors} $Q_{\new{\mu}} \in \cL(V,V)$ with
$\ran Q_{\new{\mu}} \subseteq V_\mu^\delta$, $\ran Q_{\new{\mu}}^* \supseteq \invbreve{W}_{\new{\mu}}^\delta$. \new{Moreover, it holds that
$\frac{1}{\gamma_\Delta} \leq \sup_{\mu \in \vee_\Psi} \|Q_{\new{\mu}}\|_{\cL(V,V)} \leq 2+ \frac{1}{\gamma_\Delta}$.}
\end{proposition}

\subsection{Selection of the spatial approximation spaces as finite element spaces}
We will select the spaces $W_\lambda^\delta$ from a collection ${\mathcal O}$  of finite element spaces in $V$, which collection is closed under taking (finite) sums, and for which
\be \label{x23}
\inf_{W \in {\mathcal O}} \inf_{0 \neq w \in W}  \sup_{0 \neq v \in W}  \frac{w(v)}{\|w\|_{V'} \|v\|_{V}}  >0.
\ee
Consequently, the stability conditions  \eqref{x20}--\eqref{x21} are satisfied for some $\gamma_\Delta>0$ by simply taking in Proposition~\ref{prop2},
\be \label{simple}
V_\mu^\delta:=\invbreve{W}_\mu^\delta \in {\mathcal O}.
\ee

\new{Proposition~\ref{prop3} shows that} \eqref{x23} is equivalent to uniform boundedness w.r.t.~the \new{$V$-norm} of the $H$-orthogonal projector onto $W \in {\mathcal O}$.
\new{An} example of a collection ${\mathcal O}$ \new{for which the latter is true} is given by the set of all finite element spaces $W_\lambda^\delta$ w.r.t.~quasi-uniform, uniformly shape regular conforming partitions of $\Omega$ into, say, $d$-simplices.

It is known that uniform boundedness w.r.t.~the $V$-norm of the $H$-orthogonal projector holds also true for finite element spaces w.r.t.~\emph{locally refined} partitions as long as the grading of the partitions is sufficiently mild.
\new{Recently, in~\cite{DST20},} it was shown that \new{in} $d=2$ spatial dimensions, \new{and for any polynomial order}, the collection of all conforming partitions that can be generated by \emph{newest vertex bisection} (NVB), starting from a fixed conforming initial partition $\tria_\bot$ with an assignment of the newest vertices that satisfies a so-called matching condition, is sufficiently mildly graded in the above sense.
Since the overlay of two conforming NVB partitions is a conforming NVB partition, this collection is closed under taking (finite) sums.
In other words, with this collection of finite element spaces, which we will employ in our experiments, again the choice \eqref{simple} guarantees uniform stability.

\new{Also in~\cite{DST20}, a similar result was shown in $d=2$ for red-blue-green refinement with any polynomial order.} Unfortunately, for $d>2$ such results seem not yet to be  available.

\begin{remark}[Getting $\gamma_\Delta$ close to $1$] \label{rem2} We discussed uniform boundedness w.r.t.~the $V$-norm of the $H$-orthogonal projectors onto a family of finite element spaces, which, by taking
  $V_\mu^\delta:=\invbreve{W}_\mu^\delta$ in Proposition~\ref{prop2}, yields the uniform inf-sup condition \eqref{x21} \new{for}
\emph{some} value $\gamma_\Delta>0$, and so uniform stability of the Galerkin discretizations of the saddle-point \eqref{m1}.

For proving convergence of our adaptive routine Algorithm~\ref{54}, however, we needed a value of $\gamma_\Delta>0$ that is sufficiently close to $1$.
Although in our numerical experiments, reported on in Sect.~\ref{Snumerics}, with continuous piecewise linear finite element spaces generated by conforming NVB and $V_\mu^\delta=\invbreve{W}_\mu^\delta$, the adaptive routine is $r$-linearly converging, there is no guarantee that $1-\gamma_\Delta$ is sufficiently small.

Restricting to quasi-uniform partitions, \new{it can be shown} that $1-\gamma_\Delta$ can be made arbitrarily small by taking the mesh underlying $V_\mu^\delta$ to be a sufficiently deep, but fixed \emph{refinement} of the mesh underlying $\invbreve{W}_\mu^\delta$. One may conjecture that the same result holds true for sufficiently mildly graded locally refined meshes. \comment{Further technical details omitted.}
\end{remark}

\subsection{Best possible rates}\label{sec:rates}
Although \new{we} have not proved it, we expect that the sequence of approximations generated by our adaptive Algorithm~\ref{54} is not only $r$-linearly converging, but, ignoring data oscillation, that it is a sequence of approximations \new{from} the family $(X^\delta)_{\delta \in \Delta}$ that converges with the best possible rate.
In this subsection, \new{under some (mild) smoothness conditions on the solution $u$,} we show that \new{for} our selection of the $(X^\delta)_{\delta \in \Delta}$ this best possible rate \emph{equals} the rate of  best approximation to the solution of the corresponding stationary problem from the spatial finite element spaces w.r.t.~the $V$-norm.

Consider a family of spaces $X^\delta=\sum_{\lambda \in \vee_\Sigma} \sigma_\lambda \otimes W_\lambda^\delta$ that satisfies \eqref{tree-in-time1}, with the $W_\lambda^\delta$ selected from a collection of finite element spaces ${\mathcal O}$ that in any case contains all \new{spaces} that correspond to uniform refinements of some initial partition of $\Omega$.
Let $\Sigma$ a collection of wavelets of order $d_t$, and assume that the finite element spaces are of order $d_x$.
When for each $X^\delta$, the space $Y^\delta$ is selected as in Proposition~\ref{prop2}, then the combination of \eqref{m6b} and the analysis from \cite[Sect.~7.1]{247.15}  shows that
if the exact solution $u$ of our parabolic problem satisfies the mixed regularity condition $u \in H^{d_t}(I)\otimes H^{d_x}(\Omega)$, then
a suitable (non-adaptive) choice of the spaces $W_\lambda^\delta$ yields a sequence of solutions $u^{\hat{\delta} \delta} \in X^\delta$ (for arbitrary $Y^{\hat{\delta}} \supset Y^\delta$)
 of the modified discretized saddle-point from Sect~\ref{Smod}, for which
$$
\|u-u^{\hat{\delta} \delta}\|_X \lesssim (\dim X^\delta)^{-\min(d_t-1,\frac{d_x-1}{d})}.
$$
Note that for $d_t-1 \geq \frac{d_x-1}{d}$, the rate $\frac{d_x-1}{d}$ equals the best rate in the $V$-norm 
when the finite element spaces are employed for solving the corresponding \emph{stationary} problem, which is posed on \new{$\Omega \subset \R^d$ instead of on $I \times \Omega \subset \R^{d+1}$}.

For an optimal \emph{adaptive} choice of the \new{finite element spaces} $W_\lambda^\delta$ \new{from a} sufficiently `rich' collection \new{that contains}  locally refined partitions, as the collection of all conforming NVB partitions, it can be expected that the rate $\min(d_t-1,\frac{d_x-1}{d})$ is realized under \emph{much} milder regularity conditions on $u$. \comment{Presentation shortened.}

\subsection{Preconditioners}
\label{sec:preconditioners}
Our adaptive \new{algorithm for} the parabolic problem requires optimal preconditioners for ${E_Y^\delta}' A E_Y^\delta$ and $S^{\udelta \delta}$,
\new{see \eqref{kappadelta} and \eqref{x28}.
That is,} for both $Z=Y$ and $Z=X$ and $\delta \in \Delta$, we need operators
$K_Z^\delta={K_Z^\delta}' \in \cL({Z^\delta}',Z^\delta)$ with $h(K_Z^\delta h) \eqsim \|h\|_{{Z^\delta}'}^2$ ($h \in {Z^\delta}'$), which moreover should be applicable at linear cost.

To construct these preconditioners, for $Z\in \{Y,X\}$ we will select a symmetric, bounded, and coercive bilinear form on $Z \times Z$, and after selecting \emph{some} basis for $Z^\delta$, we will construct a matrix ${\bf K}_Z^\delta={{\bf K}_Z^\delta}^\top$ that can be applied in linear complexity, and that is \emph{uniformly} spectrally equivalent to the inverse of the \emph{stiffness matrix} corresponding to this bilinear form (being the matrix representation of the linear mapping $Z^\delta \rightarrow {Z^\delta}'$ defined by the bilinear form w.r.t.~the basis for $Z^\delta$ \new{and} the corresponding dual basis for ${Z^\delta}'$).
Then $K_Z^\delta \in \cL({Z^\delta}',Z^\delta)$, defined as the operator whose  matrix representation is ${\bf K}_Z^\delta$ w.r.t.~the aforementioned bases of ${Z^\delta}'$ and $Z^\delta$, is the preconditioner that satisfies our needs.

Notice that the choice of the basis for $Z^\delta$ is irrelevant.
Indeed, denoting the aforementioned stiffness matrix as ${\bf C}_Z^\delta$ with corresponding operator $C_Z^\delta={C_Z^\delta}' \in \Lis(Z^\delta,{Z^\delta}')$, one may verify that
\begin{align*}
\|K_Z^\delta\|_{\cL({Z^\delta}',Z^\delta)} \|(K_Z^\delta)^{-1}\|_{\cL(Z^\delta,{Z^\delta}')}
& \eqsim \|K_Z^\delta C_Z^\delta\|_{\cL(Z^\delta,Z^\delta)} \|(K_Z^\delta C_Z^\delta)^{-1}\|_{\cL(Z^\delta,{Z^\delta})}\\
&=\frac{\lambda_{\max}({\bf K}_Z^\delta {\bf C}_Z^\delta)}{\lambda_{\min}({\bf K}_Z^\delta {\bf C}_Z^\delta)}.
\end{align*}

\subsubsection{Preconditioner at the `test side'}  \label{sec:precY} Let $\new{Z=Y}$. Since $\Psi$ is an orthonormal basis for $L_2(I)$, any $y \in Y$ \new{equals} $\sum_{\mu \in \vee_\Psi} \psi_\mu \otimes v_\mu$ for some $v_\mu \in V$ with $\sum_{\mu \in \vee_\Psi} \|v_\mu\|_V^2<\infty$.
Taking as bilinear form on $Y \times Y$ simply the scalar product on $Y \times Y$, we have
$$
\langle \sum_{\mu_1 \in \vee_\Psi} \psi_{\mu_1} \otimes v^{(1)}_{\mu_1}, \sum_{\mu_2 \in \vee_\Psi} \psi_{\mu_2} \otimes v^{(2)}_{\mu_2}\rangle_Y=
\sum_{\mu \in \vee_\Psi} \langle v^{(1)}_{\mu},v^{(2)}_{\mu}\rangle_V.
$$
Equipping $Y^\delta=\sum_{\mu \in \vee_\Psi} \psi_\mu \otimes V_\mu^\delta$ with a basis of type $\cup_{\mu \in \vee_\Psi}  \psi_\mu \otimes \Phi_\mu^\delta$, the resulting stiffness matrix reads as $\blockdiag[{\bf A}_\mu^\delta]_{\mu \in \vee_\Psi}$, where ${\bf A}_\mu^\delta=\langle \Phi_\mu^\delta, \Phi_\mu^\delta\rangle_V$ is the stiffness matrix of $\langle\cdot,\cdot\rangle_V$ w.r.t.~$\Phi_\mu^\delta$.
Selecting ${\bf K}_\mu^\delta \eqsim ({\bf A}_\mu^\delta)^{-1}$, the matrix representation of the optimal preconditioner reads as
$$
{\bf K}_Y^\delta=\blockdiag[{\bf K}_\mu^\delta]_{\mu \in \vee_\Psi}.
$$
It is well-known that when $V_\mu^\delta$ is a finite element space, possibly w.r.t.~a locally refined partition, suitable ${\bf K}_\mu^\delta$ of \emph{multi-grid} type are available. These ${\bf K}_\mu^\delta$ can be applied in linear complexity, and so can ${\bf K}_Y^\delta$.

\new{\begin{remark} \textcolor{black}{To show, in Theorem~\ref{thm:main}, \new{$r$-linear convergence of}
  our adaptive Algorithm~\ref{54}, \new{we require the constant~$\kappa_\Delta$ from~\eqref{kappadelta} to be sufficiently small,} i.e., the eigenvalues of \new{the} $K_Y^\delta {E_Y^\delta}' A E_Y^\delta$ \new{to be} sufficiently close to $1$.
Given an initial optimal, self-adjoint, \new{coercive} preconditioner $K_Y^\delta$, and some upper and lower bounds on the spectrum of the preconditioned system, one can satisfy the latter condition by polynomial acceleration using Chebychev polynomials of sufficiently high degree.
  In our numerical experiments, \new{this `acceleration' turned out to be unnecessary}.}
\end{remark}}

\subsubsection{Preconditioner at the `trial side'} \label{sec:precX}
L\new{et $Z=X$.}
The preconditioner presented in this section is inspired by constructions of preconditioners in \cite{12.5,234.7}
for parabolic problems discretized on a tensor product of temporal and spatial spaces.

\new{As} $\Sigma$ and $\{2^{-|\lambda|} \sigma_\lambda \colon \lambda \in \vee_\Sigma\}$ \new{are} Riesz bases for $L_2(I)$ and $H^1(I)$, any $x \in X$ \new{equals} $\sum_{\lambda \in \vee_\Sigma} \sigma_\lambda \otimes w_\lambda$ for some $w_\lambda \in V$ with
$\sum_{\lambda \in \vee_\Sigma} \|w_\lambda\|_V^2+4^{|\lambda|}  \|w_\lambda\|_{V'}^2<\infty$, and
$$
\langle \sum_{\lambda_1 \in \vee_\Sigma} \sigma_{\lambda_1} \otimes w^{(1)}_{\lambda_1},
 \sum_{\lambda_2 \in \vee_\Sigma} \sigma_{\lambda_2} \otimes w^{(2)}_{\lambda_2}\rangle
:=
 \sum_{\lambda \in \vee_\Sigma} \langle w^{(1)}_\lambda ,w^{(2)}_\lambda \rangle_V
+4^{|\lambda|}  \langle w^{(1)}_\lambda ,w^{(2)}_\lambda \rangle_{V'}
$$
is a symmetric, bounded, \new{coercive} bilinear form on $X \times X$.
Equipping $X^\delta=\sum_{\lambda \in \vee_\Sigma} \sigma_\lambda \otimes W_\lambda^\delta$ with a basis of type $\bigcup_{\lambda \in \vee_\Sigma} \sigma_\lambda \otimes \Phi_\lambda^\delta$, the resulting stiffness matrix reads
$$
\blockdiag[{\bf A}_\lambda^\delta+4^{|\lambda|} \langle \Phi_\lambda^\delta, \Phi_\lambda^\delta\rangle_{V'}]_{\lambda \in \vee_\Sigma}
$$
where ${\bf A}_\lambda^\delta=\langle \Phi_\lambda^\delta, \Phi_\lambda^\delta\rangle_V$.

Thanks to our assumption \eqref{x23}, for $u \in W_\lambda^\delta$ it holds that $\|u\|_{V'} \new{\eqsim} \sup_{0 \neq w \in W_\lambda^\delta} \frac{\langle u,w\rangle}{\|w\|_V}$.
With ${\bf u}$  denoting the representation of $u$ w.r.t.~$\Phi_\lambda^\delta$, we have
$$
\sup_{0 \neq w \in W_\lambda^\delta} \frac{\langle u,w\rangle}{\|w\|_V}=
\|({\bf A}_\lambda^\delta)^{-\frac12} {\bf M}_\lambda^\delta {\bf u}\|,
$$
where ${\bf M}_\lambda^\delta=\langle \Phi_\lambda^\delta, \Phi_\lambda^\delta\rangle$, so that
$$
 \langle \Phi_\lambda^\delta, \Phi_\lambda^\delta\rangle_{V'} \new{\eqsim}
{\bf M}_\lambda^\delta ({\bf A}_\lambda^\delta)^{-1} {\bf M}_\lambda^\delta.
$$

Since both ${\bf A}_\lambda^\delta$ and ${\bf M}^\delta_\lambda$ are \new{symmetric positive definite},
\cite[Thm.~4]{242.817} shows
\begin{align*}
{\textstyle \frac{1}{2}}\big( {\bf A}_\lambda^\delta\!+\!4^{|\lambda|} {\bf M}^\delta_\lambda ({\bf A}^\delta_\lambda)^{-1}  {\bf M}^\delta_\lambda\big)
\!\leq \! ({\bf A}_\lambda^\delta\!+\!2^{|\lambda|}{\bf M}^\delta_\lambda) ({\bf A}_\lambda^\delta)^{-1} ({\bf A}_\lambda^\delta\!+\!2^{|\lambda|}{\bf M}^\delta_\lambda)
\!\leq\!
{\bf A}_\lambda^\delta\!+\!4^{|\lambda|} {\bf M}^\delta_\lambda ({\bf A}^\delta_\lambda)^{-1}  {\bf M}^\delta_\lambda.
\end{align*}
Now assuming that
\be \label{robustmg}
{\bf K}^\delta_{\lambda} \eqsim ({\bf A}_\lambda^\delta+2^{|\lambda|}{\bf M}^\delta_\lambda)^{-1},
\ee
we infer that
$$
{\bf K}_X^\delta=
\blockdiag\Big[{\bf K}^\delta_{\lambda} {\bf A}_\lambda^\delta {\bf K}^\delta_{\lambda}\Big]_{\new{\lambda \in \vee_\Sigma}}
$$
is the matrix representation of an optimal preconditioner.

\begin{remark}
  Notice that \eqref{robustmg} requires an optimal preconditioner of a discretized reaction-diffusion equation that is robust \new{w.r.t.}~the size of the (constant) reaction term.
  In \cite{241.1} it was shown that, under a `full-regularity' assumption, for quasi-uniform meshes multiplicative multi-grid yields such a \new{preconditioner, whose} application can \new{even} be performed at linear cost. Although we expect that
  \new{this assumption can be avoided using}
  the theory of subspace correction methods, and furthermore that the optimality, robustness and linear complexity result extends to locally refined meshes, proofs of such extensions seem not to be available.
\end{remark}

\section{A concrete realization} \label{S6}
\subsection{The collection ${\mathcal O}$ of finite element spaces, and the mapping $\delta \rightarrow \udelta$}\label{Smapping}
We further specify the collection ${\mathcal O}$ of finite element spaces, construct a linearly independent set in $H^1_{0,\Gamma_D}(\Omega)$ known as the hierarchical basis, and equip it with a tree structure such that there exists a 1-1 correspondence between the finite element spaces in ${\mathcal O}$, and the spans of subsets of the hierarchical basis that form trees.

With this specification of ${\mathcal O}$, there will be a 1-1 correspondence between the spaces $X^\delta=\sum_{\lambda \in \sigma_\lambda} \sigma_\lambda \otimes W_\lambda^\delta$ with $W_\lambda^\delta \in {\mathcal O}$ that satisfy \eqref{tree-in-time1}, and the spans of collections of tensor products of wavelets $\sigma_\lambda$ and hierarchical basis functions whose sets of index pairs are \emph{lower}, also known as \emph{downward closed.}
Given such a $X^\delta$, we will define $X^\udelta$ by a certain enlargement the lower set.
\smallskip

For $d \geq 2$, let $\bbT$ be the family of all \emph{conforming} partitions of a polytope $\Omega \subset \R^d$ into (closed) $d$-simplices that can be created by NVB
starting from some given conforming initial partition $\tria_\bot$ with an assignment of the newest vertices that satisfies the matching condition, see \cite{249.87}.
We define a \emph{partial order} on $\bbT$ by writing $\tria \preceq \tilde{\tria}$ when $\tilde{\tria}$ is a refinement of $\tria$.

With some small adaptations that we leave to the reader, in the following
the case $d=1$ can be included by letting $\bbT$ \new{be the family of partitions} of $\Omega$ into (closed) \new{subintervals  constructed} by bisections from $\tria_\bot=\{\Omega\}$ such that the generations of any two neighbouring subintervals in any $\tria \in \bbT$ differ by not more than one.

The collection ${\mathcal O}$ that we will consider is formed by the spaces $W=W_\tria$ of \emph{continuous piecewise linears} w.r.t.~ $\tria \in \bbT$, zero on a possible Dirichlet boundary $\Gamma_D$ being the union of $\partial T \cap \partial\Omega$ for some $T \in \tria_\bot$.
We expect that generalizations to finite element spaces of higher order do not impose essential difficulties.

For $T \in \mathfrak{T}:=\cup_{\tria \in \bbT} \{T\colon T \in \tria\}$, we set $\gen(T)$ to be the number of bisections needed to create $T$ from its `ancestor' $T' \in \tria_\bot$.
With $\mathfrak{N}$ being the set of all vertices (or nodes) of all $T \in\mathfrak{T}$,
for $\nu \in \mathfrak{N}$ we set $\gen(\nu):=\min\{\gen(T)\colon T \in \mathfrak{T},\,\nu \in T\}$.

Any $\nu \in \mathfrak{N}$ with $\gen(\nu)>0$ is the midpoint of an edge of one or more $T  \in \mathfrak{T}$ with $\gen(T)=\gen(\nu)-1$. The vertices $\tilde{\nu}$ of these $T$ with $\gen(\tilde{\nu})=\gen(\nu)-1$ are defined as the parents of $\nu$. We denote this relation between a parent $\tilde{\nu}$ and a child $\nu$ by
$\tilde{\nu} \triangleleft_\mathfrak{N} \nu$, see Figure~\ref{parentchild}.
\begin{figure}
\begin{center}
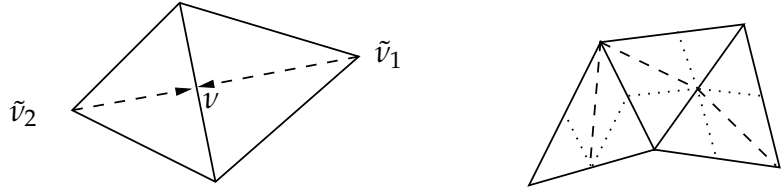
\end{center}
  \caption{$\tilde{\nu}_1,\tilde{\nu}_2 \triangleleft_\mathfrak{N} \nu$ (left picture); and $\tria$ and its refinement $\tria^{d +}$ (for $d=2$) (right picture).}
  \label{parentchild}
\end{figure}
Vertices $\nu \in \mathfrak{N}$ with $\gen(\nu)=0$ have no parents.

 An (essentially) non-overlapping partition $\tria$ of $\overline{\Omega}$ into $T  \in \mathfrak{T}$ is in
 $\bbT$ \emph{if and only if} the set $N_\tria$ of vertices of all $T \in \tria$ forms a \emph{tree}, meaning that it contains all $\nu \in \mathfrak{N}$ with $\gen(\nu)=0$ as well as all parents of any $\nu \in N_\tria$ with $\gen(\nu)>0$, cf.~\cite{64.59} for the $d=2$ case.

\begin{definition}
For any $\tria \in \bbT$, we define $\tria^{d +} \in \bbT$ (denoted as $\tria^{++}$ in \cite{64.59} for the $d=2$ case)
by replacing any $T \in \tria$ by its $2^d$ `descendants' of the $d$th generation, see Figure~\ref{parentchild}.
\end{definition}
 Since this refinement adds exactly one vertex at the midpoint on any edge of all $T \in \tria$, one infers that indeed $\tria^{d +} \in \bbT$.
The corresponding tree $N_{\tria^{d +}}$ is created from $N_\tria$ by the addition of all descendants up to generation $d$ of all $\nu \in N_\tria$.\footnote{The addition of only all children of all $\nu \in N_\tria$ yields a tree only if $\tria$ is a uniform partition.}

 For $\nu \in \mathfrak{N}$, we set $\phi_\nu$ as the continuous piecewise linear function w.r.t.~the \emph{uniform partition} $\{T \in  \mathfrak{T}\colon \gen(T)=\gen(\nu)\} \in \bbT$, which function is $1$ at $\nu$ and $0$ at all other vertices of this partition.
 Setting $\mathfrak{N}_0:=\mathfrak{N}\setminus \Gamma_D$ and, for any $\tria \in \bbT$, $N_{\tria,0}:=N_\tria \setminus \Gamma_D$,
 the collection $\{\phi_\nu\colon \nu \in \mathfrak{N}_0\}$ is known as the \emph{hierarchical basis}, and for any $\tria \in \bbT$, it holds that $W_\tria=\Span\{\phi_\nu\colon \nu \in N_{\tria,0}\}$.

 With above specification of the collection ${\mathcal O}$ of finite element spaces, there exists a 1-1 correspondence between the spaces $\sum_{\lambda \in \vee_\Sigma} \sigma_\lambda \otimes W_\lambda^\delta$ with $W_\lambda^\delta \in {\mathcal O}$ that satisfy \eqref{tree-in-time1},
 and the spaces of the form
 \be \label{eqn:Xdelta}
X^\delta=\Span\{\sigma_\lambda \otimes \phi_\nu\colon (\lambda,\nu) \in I_{\delta,0}:=I_\delta \setminus  (\vee_\Sigma \times \Gamma_D)\}
 \ee
 for some finite
 $I_\delta \subset \vee_\Sigma \times \mathfrak{N}$ being a \emph{lower set} in the sense
  \be \label{eqn:Xdeltab}
 (\lambda,\nu) \in I_\delta \text{ and} \left\{\begin{array}{lcl}
 \tilde{\lambda} \triangleleft_\Sigma \lambda & \Longrightarrow & (\tilde{\lambda},\nu) \in I_\delta,\\
 \tilde{\nu} \triangleleft_\mathfrak{N} \nu \text{ or } \gen(\tilde{\nu})=0 & \Longrightarrow & (\lambda,\tilde{\nu}) \in I_\delta.
 \end{array}\right.
 \ee

For above specification of $X^\delta$, from Proposition~\ref{prop2} with the specification \eqref{simple} one infers that the corresponding space \new{$Y^\delta$ is given by}
\be \label{eqn:Ydelta}
Y^\delta=\Span\{\psi_\mu \otimes \phi_\nu\colon (\mu,\nu) \in I_{\delta,0}^Y\},
 \ee
 where
 \be \label{eqn:Ydeltab}
  I_{\delta,0}^Y:=\{(\mu,\nu)\colon \exists (\lambda,\nu) \in I_{\delta,0},\, \mu \in \vee_\Psi,\,|\mu|=|\lambda|,\,|\cS_\Psi(\mu) \cap \cS_\Sigma(\lambda)|>0\}
  \ee
  which index set is a lower set.

 \begin{remark}[Complexity of matrix-vector multiplications]\label{rem:matvec}
 The fact that the index sets of the bases for $X^\delta$ and $Y^\delta$ are lower sets is the key
   \new{to computing} residuals of the system $S^{\udelta \delta} u^\delta=f^\delta$ (\eqref{52}) in ${\mathcal O}(\dim X^\delta)$ operations.
 \comment{Presentation shortened.}
  The algorithm that realizes this complexity  makes a clever use of multi- to single-scale transformations alternately in time and space. In a `uniform' sparse-grid setting, i.e., without `local refinements', this algorithm was introduced in \cite{18.83}, and it was later extended to general lower sets in \cite{171.7}. The definition of a lower set in \cite{171.7}, there called multi-tree, is more restrictive than our current definition that allows more localized refinements, and details about the matrix-vector multiplication and a proof of its optimal computational complexity can be found in \cite{306.6}.
 \end{remark}

 \begin{definition}\label{def:Xudelta} Given $X^\delta=\Span\{\sigma_\lambda \otimes \hat{\phi}_\nu\colon (\lambda,\nu) \in I_{\delta,0}\}$ for some  lower set $I_\delta \subset \vee_\Sigma \times \mathfrak{N}$, we define the lower set $I_{\udelta}$, and with that $X^{\udelta}$,
  by adding, for each $(\lambda,\nu) \in I_\delta$
  and any child $\tilde{\lambda}$ of $\lambda$ and any descendant $\tilde{\nu}$ of $\nu$ up to generation $d$, all pairs
  $(\tilde{\lambda},\nu)$ and $(\lambda,\tilde{\nu})$ to $I_\delta$.
  \end{definition}

\subsection{The collection $\Theta_\delta$ such that $X^\udelta=X^\delta \oplus \Theta_\delta$}
Recall that for the bulk chasing process we need an `$X$-stable' basis $\Theta_\delta$ that spans an `$X$-stable' complement space of $X^\delta$ in $X^\udelta$, i.e., a collection that satisfies \eqref{m17}. For that goal we define a \emph{modified hierarchical basis} $\{\hat{\phi}_\nu\colon \nu \in \mathfrak{N}_0\}$ by $\hat{\phi}_\nu:=\phi_\nu$ when $\gen(\nu)=0$, and
$$
\hat{\phi}_\nu:=\phi_\nu-\frac{\sum_{\{\tilde{\nu} \in \mathfrak{N}_0\colon \tilde{\nu}\triangleleft_\mathfrak{N} \nu\}}
\frac{\int_\Omega \,\phi_\nu dx}{\int_\Omega \,\phi_{\tilde{\nu}} dx}
\phi_{\tilde{\nu}}}
{\# \{\tilde{\nu} \in \mathfrak{N}\colon \tilde{\nu}\triangleleft_\mathfrak{N} \nu\}}
$$
otherwise. Notice that for those $\nu$ with $\gen(\nu)>0$ that have all their parents not on $\Gamma_D$ it holds that $\int_\Omega \hat{\phi}_\nu\,dx=0$, i.e., $\hat{\phi}_\nu$ has a \emph{vanishing moment}, and furthermore that for \new{any $\tria \in \bbT$},
$$
W_\tria=\Span\{\phi_\nu\colon \nu \in N_{\tria,0}\}=\Span\{\hat{\phi}_\nu\colon \nu \in N_{\tria,0}\},
$$
 and thus for any lower set $I_\delta \subset \vee_\Sigma \times \mathfrak{N}$,
$$
X^\delta=\Span\{\sigma_\lambda \otimes \hat{\phi}_\nu\colon (\lambda,\nu) \in I_{\delta,0}\}=\Span\{\sigma_\lambda \otimes \phi_\nu\colon (\lambda,\nu) \in I_{\delta,0}\}.
$$
Moreover, for any $\tria \in \bbT$,
the basis transformation from the modified to unmodified hierarchical basis for $W_\tria$ can be applied in linear complexity traversing from the leaves to the roots.

Given $\delta$, the collection $\Theta_\delta$ will be the set of properly normalized functions $\sigma_\lambda \otimes \hat{\phi}_\nu$ for $(\lambda,\nu) \in I_{\udelta,0}\setminus I_{\delta,0}$. In order to demonstrate \eqref{m17}, we have to impose some gradedness assumption on the lower sets $I_\delta$.

  \begin{definition} \label{def:gradedness} The \emph{gradedness constant} of a lower set $I_\delta \subset \vee_\Sigma \times \mathfrak{N}$ is the smallest $L_\delta \in \N$ such that for all $(\lambda,\nu) \in I_\delta$ for which $\nu$ has an ancestor $\tilde{\nu} \in \mathfrak{N}$ with $\gen(\nu)-\gen(\tilde{\nu})=L_\delta$, it holds that  $(\breve{\lambda},\tilde{\nu}) \in I_\delta$ for any child $\breve{\lambda} \in \vee_\Sigma$ of $\lambda$.
   \end{definition}

\begin{remark}[Uniform boundedness of the gradedness constants] Under the (unproven) assumption that our adaptive method creates a sequence of spaces $X^\delta$ which are quasi-optimal for the approximation of the solution of the the parabolic PDE, one may hope that these spaces have a \emph{uniformly bounded gradedness constant}, unless (locally) the solution $u$ is extremely more smooth as function of $t$ than as function of the spatial variables.

To see this, \new{for some constant $L \in \N$, consider the non-adaptive sparse grid index sets of the form
   $\{(\lambda,\nu) \in \vee_\Sigma \times \mathfrak{N}\colon L |\lambda| +\gen(\nu) \leq N\}$,
   which are appropriate when the behaviour of $u$ as function of $t$ and  of the spatial variables is globally similar.
   Their gradedness constant is $L$.
   In the corresponding trial spaces, the smallest spatial resolution equals the smallest temporal resolution to the power $L/d$. 
   So only when such a polynomial decay of the spatial resolution as function of the temporal resolution does not suffice for a proper approximation of $u$, a uniformly bounded gradedness constant cannot be expected.}
 \end{remark}

   \begin{proposition} \label{stableset} For $(\lambda,\nu) \in \vee_\Sigma \times \mathfrak{N}_0$, let
   $e_{\lambda \mu}:=1/\sqrt{2^{(\frac{2}{d}-1)\gen(\nu)}+4^{|\lambda|} 2^{(-\frac{2}{d}-1)\gen(\nu)}}$ and
   $\theta_{\lambda \nu}:=e_{\lambda \mu}\sigma_\lambda \otimes \hat{\phi}_\nu$.
 For any $\delta \in \Delta$, let $\Theta_\delta:=\{\theta_{\lambda \nu}\colon (\lambda,\nu) \in  J_\delta:=I_{\udelta,0}\setminus I_{\delta,0}\}$. Then
  $X^\delta\oplus \Span \Theta_\delta=X^\udelta$, and there exist constants $0<m_\delta \leq M_\delta$, \emph{only dependent} on the gradedness constant $L_\delta$, such that for any $z \in X^\delta$ and ${\bf c}=(c_{\lambda \nu})_{(\lambda,\nu) \in  I_{\udelta,0}\setminus I_{\delta,0}} \subset \R$,
 $$
 m_\delta (\|z\|^2_X+\|{\bf c}\|^2) \leq \|z+{\bf c}^\top \Theta_\delta\|_X^2 \leq M_\delta (\|z\|^2_X+\|{\bf c}\|^2).
 $$
   \end{proposition}
So under the mild assumption that the gradedness constants of the sets $X^\delta$ that we encounter are uniformly bounded, \new{condition} \eqref{m17} is satisfied.

   \begin{proof}
   Setting $c_{\lambda \nu}:=0$ when $(\lambda,\nu) \not\in I_{\udelta,0}\setminus I_{\delta,0}$,
   and writing $z =\sum_{\lambda \in \vee_\Sigma} \sigma_\lambda \otimes w_\lambda$ where  $w_\lambda \in \Span\{\hat{\phi}_\nu\colon (\lambda,\nu) \in I_{\delta,0}\}$, from
    $\Sigma$ and $\{2^{-|\lambda|} \sigma_\lambda \colon \lambda \in \vee_\Sigma\}$ being Riesz bases for $L_2(I)$ and $H^1(I)$,
    and ${\bf c}^\top \Theta_\delta=\sum_\lambda \sigma_\lambda\otimes \sum_\nu e_{\lambda \mu}
    c_{\lambda \nu} \hat{\phi}_\nu$, an application of Lemma~\ref{lemmaatje} given below shows that

    \begin{align*}
    &\|z+{\bf c}^\top  \Theta_\delta\|_X^2 \eqsim
    \sum_\lambda \big\{\|w_\lambda+\sum_{\nu} e_{\lambda \mu} c_{\lambda \nu} \hat{\phi}_\nu\|_{H^1(\Omega)}^2+4^{|\lambda|} \|w_\lambda+\sum_{\nu} e_{\lambda \mu} c_{\lambda \nu} \hat{\phi}_\nu\|_{H^1_{0,\Gamma_D}(\Omega)'}^2\big\}
  \\
    &\eqsim \!\big\{\! \sum_\lambda \|w_\lambda\|_{H^1(\Omega)}^2\!\!+\! 4^{|\lambda|} \|w_\lambda\|_{H^1_{0,\Gamma_D}(\Omega)'}^2 \!+\!\sum_\nu \big(2^{(\frac{2}{d}-1)\gen(\nu)}\!\!+\!4^{|\lambda|} 2^{(-\frac{2}{d}-1)\gen(\nu)} \big) |e_{\lambda \mu} c_{\lambda \nu} |^2\big\}\\
    &=\sum_\lambda \|w_\lambda\|_{H^1(\Omega)}^2+ 4^{|\lambda|} \|w_\lambda\|_{H^1_{0,\Gamma_D}(\Omega)'}^2 +
    \sum_\nu |c_{\lambda \nu} |^2
    \eqsim \|z\|_X^2+\|{\bf c}\|^2,
    \end{align*}
    with the $\eqsim$-symbol in the second line dependent on the gradedness constant.
   \end{proof}

   \new{The proof of Proposition~\ref{stableset} was based on the following lemma.}

   \begin{lemma} \label{lemmaatje} For $\tilde{\tria} \in \bbT$, and either $\bbT \ni \tria \preceq \tilde{\tria}$ and $v \in W_\tria$, or $\tria=\emptyset$, $N_{\tria,0}:=\emptyset$, and $v=0$,  and scalars $(d_\nu)_{\nu \in N_{\tilde{\tria},0}\setminus N_{\tria,0}}$, it holds that
   \begin{align} \label{equiv1}
   \|v+\sum_{\nu} d_\nu \hat{\phi}_\nu\|_{H^1(\Omega)}^2 &\eqsim \|v\|_{H^1(\Omega)}^2+\sum_{\nu} 2^{(\frac{2}{d}-1)\gen(\nu)} |d_\nu|^2\\  \label{equiv2}
     \|v+\sum_{\nu} d_\nu \hat{\phi}_\nu\|_{H^1_{0,\Gamma_D}(\Omega)'}^2 &\eqsim \|v\|_{H^{-1}(\Omega)}^2+\sum_{\nu} 2^{(-\frac{2}{d}-1)\gen(\nu)} |d_\nu|^2
   \end{align}
   with the constants hidden in the $\eqsim$-symbols only dependent on $M_{\tilde{\tria} \tria}:=\max\{\gen(\tilde{T})-\gen(T)\colon \tilde{\tria} \ni  \tilde{T} \subset T \in \tria\}$ or $M_{\tilde{\tria} \tria}:=\max\{\gen(\tilde{T})\colon
   \tilde{T} \in \tilde{\tria}\}$ for $\tria=\emptyset$.
   \end{lemma}

   \begin{proof} Once the equivalences are shown uniformly in any $\tria \preceq \tilde{\tria} $ \new{with} $M_{\tilde{\tria} \tria}=1$, \new{repeated application} shows them for the general case, with \new{constants only} dependent on $M_{\tilde{\tria} \tria}$. So in the following, it suffices to consider the case that $M_{\tilde{\tria} \tria}=1$. The case $\tria=\emptyset$ is easy, so we will consider the case that $\tria \in \bbT$.

   Let $\Phi_{\tilde{\tria}}=\{\phi_{\tilde{\tria},\nu}\colon \nu \in N_{\tilde{\tria},0}\}$ denote the standard nodal basis for $W_{\tilde{\tria}}$.
   For any weight function $0<w_{\tilde{\tria}} \in \prod_{T \in \tilde{\tria}} P_0(T)$, with $\|\cdot\|_{L_{2,w_{\tilde{\tria}}}(\Omega)}:=\|w_{\tilde{\tria}}^{\frac12} \cdot\|_{L_2(\Omega)}$ it holds that $\|\sum_\nu c_\nu \phi_{\tilde{\tria},\nu}\|_{L_{2,w_{\tilde{\tria}}}(\Omega)}^2 \eqsim \sum_\nu |c_\nu|^2 \|\phi_{\tilde{\tria},\nu}\|_{L_{2,w_{\tilde{\tria}}}(\Omega)}^2$, only dependent on the spectrum of the element mass matrix on a reference element, i.e., on the space dimension $d$, so independent of the weight function $w_{\tilde{\tria}}$. We refer to this equivalence by saying that $\Phi_{\tilde{\tria}}$ is (uniformly) stable w.r.t.~$\|\cdot\|_{L_{2,w_{\tilde{\tria}}}(\Omega)}$.

   Notice that for $\nu \in N_{\tilde{\tria},0} \setminus  N_{\tria,0}$, it holds that $\phi_{\tilde{\tria},\nu}=\phi_{\nu}$.
   W.r.t.~the splitting $N_{\tilde{\tria},0}=N_{\tria,0}+N_{\tilde{\tria},0}\setminus N_{\tria,0}$, the basis transformation from
   $\Phi_\tria \cup \{\hat{\phi}_{\nu}\colon \nu \in N_{\tilde{\tria},0}\setminus N_{\tria,0}\}$ to $\Phi_\tria \cup \{\phi_{\nu}\colon \nu \in N_{\tilde{\tria},0}\setminus N_{\tria,0}\}$ is of the form $\left[\begin{array}{@{}cc@{}} \identity & *\\ 0 & \identity\end{array}\right]$, and the basis transformation from the latter basis to $\Phi_{\tilde{\tria}}$ is of the form $\left[\begin{array}{@{}cc@{}} \identity & 0\\ * & \identity\end{array}\right]$.
  The entries in both non-zero off-diagonal blocks are uniformly bounded, where non-zeros can only occur for index pairs $(\nu,\tilde{\nu})$ that are vertices of the same $\tilde{T} \in \tilde{\tria}$.
     Consequently, for a family of weight functions $(w_{\tilde{\tria}})_{\tilde{\tria} \in \bbT}$ that have \emph{uniformly bounded jumps} \new{in that}
   \be \label{weight}
   \sup_{\tilde{\tria} \in \bbT} \sup_{\{T,T' \in \tilde{\tria}\colon T \cap T'\neq \emptyset\}} \frac{w_{\tilde{\tria}}|_T}{w_{\tilde{\tria}}|_{T'}}<\infty,
   \ee
 all basis transformations between the $L_{2,w_{\tilde{\tria}}}(\Omega)$-normalized bases
   $\Phi_\tria \cup \{\hat{\phi}_{\nu}\colon \nu \in N_{\tilde{\tria},0}\setminus N_{\tria,0}\}$, $\Phi_\tria \cup \{\phi_{\nu}\colon \nu \in N_{\tilde{\tria},0}\setminus N_{\tria,0}\}$ and $\Phi_{\tilde{\tria}}$ are uniformly bounded.

   Since, as we have seen, $\Phi_{\tilde{\tria}}$ is (uniformly) stable w.r.t.~$\|\cdot\|_{L_{2,w_{\tilde{\tria}}}(\Omega)}$, we conclude that also
     $\Phi_\tria \cup \{\hat{\phi}_{\nu}\colon \nu \in N_{\tilde{\tria},0}\setminus N_{\tria,0}\}$ and  $\Phi_\tria \cup \{\phi_{\nu}\colon \nu \in N_{\tilde{\tria},0}\setminus N_{\tria,0}\}$ are (uniformly) stable w.r.t.~$\|\cdot\|_{L_{2,w_{\tilde{\tria}}}(\Omega)}$.
  Because of the \emph{uniform $K$-mesh property} of $\tria \in \bbT$, examples of families of weights that satisfy \eqref{weight} are given by $(h^s_{\tilde{\tria}})_{\tilde{\tria} \in \bbT}$ for any $s \in \R$,
  where $h_{\tilde{\tria}}|_T:=2^{-\gen(T)/d} (\eqsim |T|^{1/d})$ ($T \in \tilde{\tria}$).

     \new{To show} \eqref{equiv1}, let $P_\tria\colon W_{\tilde{\tria}} \rightarrow W_\tria$ be the projector with $\ran P_\tria=W_\tria$ and $\ran (\identity-P_\tria)=\Span\{\hat{\phi}_\nu\colon \nu \in N_{\tilde{\tria},0} \setminus N_{\tria,0}\}$. Using the form of the basis \new{transform} from
   $\Phi_\tria \cup \{\phi_{\nu}\colon \nu \in N_{\tilde{\tria},0}\setminus N_{\tria,0}\}$ to
     $\Phi_\tria \cup \{\hat{\phi}_{\nu}\colon \nu \in N_{\tilde{\tria},0}\setminus N_{\tria,0}\}$, one \new{infers}
     $$
     \identity-P_\tria=J_\tria \circ (\identity-I_\tria),
     $$
     where $I_\tria$ is the nodal interpolator onto $W_\tria$, and $J_\tria$ is defined by $J_\tria \phi_{\nu}=\hat{\phi}_{\nu}$.
     Since both $\{\hat{\phi}_{\nu}\colon \nu \in N_{\tilde{\tria},0}\setminus N_{\tria,0}\}$ and $\{\phi_{\nu}\colon \nu \in N_{\tilde{\tria},0}\setminus N_{\tria,0}\}$ are uniformly stable w.r.t.~$\|h_{\tilde{\tria}}^{-1}\cdot\|_{L_2(\Omega)}$, and $\|h_{\tilde{\tria}}^{-1} \hat{\phi}_{\nu}\|_{L_2(\Omega)} \eqsim \|h_{\tilde{\tria}}^{-1} \phi_{\nu}\|_{L_2(\Omega)}$, it follows that $J_\tria$ is uniformly bounded w.r.t.~$\|h_{\tilde{\tria}}^{-1}\cdot\|_{L_2(\Omega)}$, i.e.~$ \|h_{\tilde{\tria}}^{-1} J_\tria h_{\tilde{\tria}}\|_{\cL(L_2(\Omega),L_2(\Omega))} \lesssim 1$, and so
     $$
     \|h_{\tilde{\tria}}^{-1} (\identity-P_\tria)v\|_{L_2(\Omega)} \lesssim \|h_{\tilde{\tria}}^{-1} (\identity-I_\tria)v\|_{L_2(\Omega)} \lesssim |v|_{H^1(\Omega)} \quad (v \in W_{\tilde{\tria}}).
     $$
     \new{With} the common inverse inequality $\|\cdot\|_{H^1(\Omega)} \lesssim \|h_{\tilde{\tria}}^{-1} \cdot\|_{L_2(\Omega)}$ on $W_{\tilde{\tria}}$, we \new{see} that $\identity-P_\tria$ is uniformly bounded \new{in} $H^1(\Omega)$-norm, and that on $\ran(\identity - P_\tria)$, $\|\cdot\|_{H^1(\Omega)} \eqsim \|h_{\tilde{\tria}}^{-1} \cdot\|_{L_2(\Omega)}$. The proof of \eqref{equiv1} is completed by the uniform stability of $\{\hat{\phi}_{\nu}\colon \nu \in N_{\tilde{\tria},0}\setminus N_{\tria,0}\}$ w.r.t.~$\|h_{\tilde{\tria}}^{-1}\cdot\|_{L_2(\Omega)}$, \new{using} that
     $\|h_{\tilde{\tria}}^{-1}\hat{\phi}_{\nu}\|^2_{L_2(\Omega)} \eqsim 2^{(\frac{2}{d}-1)\gen(\nu)}$.

   Moving to \eqref{equiv2}, either by $\int_\Omega \hat{\phi}_\nu \,dx=0$, or otherwise using the proximity of the Dirichlet boundary $\Gamma_D$ by an application of Poincar\'{e}'s inequality, it holds that
     $$
     |\langle \hat{\phi}_\nu,v\rangle_{L_2(\Omega)}| \lesssim 2^{-\gen(\nu)/d}\|\hat{\phi}_\nu\|_{L_2(\Omega)} |v|_{H^1(\supp \hat{\phi}_\nu)} \quad(\nu \in \mathfrak{N}_0 \setminus N_{\tria_\bot,0}).
     $$
     By using that for $T \in \tilde{\tria}$ the number of $\nu \in N_{\tilde{\tria},0}\setminus N_{\tria,0}$ for which $\supp \hat{\phi}_\nu$ has
     non-empty intersection with $T$ is uniformly bounded, and furthermore
          that $\Phi_\tria \cup \{\hat{\phi}_{\nu}\colon \nu \in N_{\tilde{\tria},0}\setminus N_{\tria,0}\}$ is uniformly stable w.r.t.~$\|h_\tria \cdot\|_{L_2(\Omega)}$, we infer that for any $z=\sum_{\nu \in N_{\tria,0}} z_\nu \phi_{\tria,\nu} \in W_{\tria}$ it holds that
         \be \label{lang}
     \begin{split}
     &\|\sum_{\nu \in N_{\tilde{\tria},0} \setminus N_{\tria,0}} d_\nu \hat{\phi}_\nu\|_{H^1_{0,\Gamma_D}(\Omega)'}=
     \sup_{0 \neq v \in H^1_{0,\Gamma_D}(\Omega)} \frac{\langle \sum_{\nu \in N_{\tilde{\tria},0} \setminus N_{\tria,0}} d_\nu \hat{\phi}_\nu,v\rangle_{L_2(\Omega)}}{\|v\|_{H^1(\Omega)}}\\
     & \lesssim \sqrt{\sum_{\nu \in N_{\tilde{\tria},0} \setminus N_{\tria,0}} |d_\nu|^2 \|h_{\tilde{\tria}}  \hat{\phi}_\nu\|_{L_2(\Omega)}^2}\\
     &\leq \sqrt{\sum_{\nu \in N_{\tria,0}} |z_\nu|^2  \|h_{\tilde{\tria}}\phi_{\tria,\nu}\|_{L_2(\Omega)}^2+
     \sum_{\nu \in N_{\tilde{\tria},0} \setminus N_{\tria,0}} |d_\nu|^2 \|h_{\tilde{\tria}}  \hat{\phi}_\nu\|_{L_2(\Omega)}^2}\\
     &\eqsim \|h_{\tilde{\tria}}(z+\sum_{\nu \in N_{\tilde{\tria},0} \setminus N_{\tria,0}}d_\nu\hat{\phi}_\nu)\|_{L_2(\Omega)}\lesssim \|z+\sum_{\nu \in N_{\tilde{\tria},0} \setminus N_{\tria,0}}d_\nu\hat{\phi}_\nu\|_{H^1_{0,\Gamma_D}(\Omega)'}
     \end{split}
     \ee
     the last inequality by application of a less common inverse inequality which proof can be found in \cite[Lemma 3.4]{249.97} for general dimensions $d$.
     From \eqref{lang} it follows that $\identity - P_\tria$ is uniformly bounded in the $H^1_{0,\Gamma_D}(\Omega)'$-norm, and also that
     $\|\sum_{\nu \in N_{\tilde{\tria},0} \setminus N_{\tria,0}} d_\nu \hat{\phi}_\nu\|^2_{H^1_{0,\Gamma_D}(\Omega)'}
     $ $\eqsim\sum_{\nu \in N_{\tilde{\tria},0} \setminus N_{\tria,0}} |d_\nu|^2 2^{(-\frac{2}{d}-1)\gen(\nu)}$, where
     we used that $\|h_{\tilde{\tria}}\hat{\phi}_{\nu}\|^2_{L_2(\Omega)} \eqsim 2^{(-\frac{2}{d}-1)\gen(\nu)}$. The proof of \eqref{equiv2} is completed.
   \end{proof}

 \subsection{The wavelet collections $\Sigma$ and $\Psi$} \label{Sconcrete2}
As wavelet basis $\Sigma=\{\sigma_\lambda\colon \lambda \in \vee_\Sigma\}$ we select the three-point hierarchical basis illustrated in Figure~\ref{hb}.
 This basis is known to be a Riesz basis for $L_2(I)$, and, after re-normalization, for $H^1(I)$ (see \cite{249.66}). It also satisfies the other assumptions made in Sect.~\ref{Swavsintime}. The wavelets up to level $\ell$ span the space of continuous piecewise linear functions on $I$ w.r.t.~the uniform partition into $2^\ell$ subintervals.
 \begin{figure}
\begin{center}
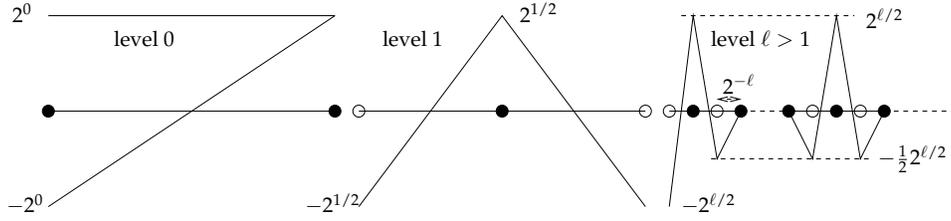
\end{center}
\caption{Three-point hierarchical basis $\Sigma$. On level $0$ there are two wavelets, and on level $1$ there is one wavelet, whose parents are both wavelets on level $0$. On each level $\ell>1$ there are
 $2^{\ell-1}$ wavelets, among them near each boundary one boundary-adapted wavelet, where each wavelet has one parent being the wavelet on level $\ell-1$ whose support includes the support of its child
(so $\cS_\Sigma(\lambda)$ can be taken equal to $\supp \sigma_\lambda$). All but one wavelets have one (bdr.~wav) or two vanishing moments. }
\label{hb}
\end{figure}

 As wavelet basis $\Psi=\{\psi_\mu\colon \mu \in \vee_\Psi\}$ we take the orthonormal (discontinuous) piecewise linear wavelets, see Figure~\ref{orthowavs}. The wavelets up to level $\ell$ span the space of (discontinuous) piecewise linear functions on $I$ w.r.t.~the uniform partition into $2^\ell$ subintervals.
 \begin{figure}
\begin{center}
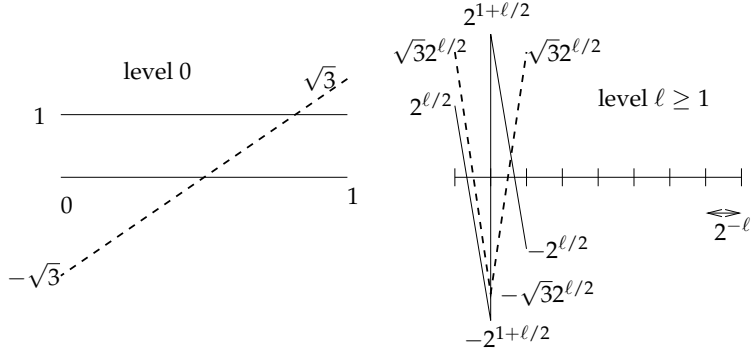
\end{center}
\caption{$L_2(I)$-orthonormal (discontinuous) piecewise linear wavelet basis $\Psi$. On level $\ell=0$ there are 2 wavelets. On each level $\ell \geq 1$ there are $2^\ell$ wavelets of two types, each of them having 2 parents being the wavelets on level $\ell-1$ whose supports include the supports of their children (so $\cS_\Psi(\mu)$ can be taken equal to $\supp \psi_\mu$). The wavelets on level 0 have either 0 or 1 vanishing moment, all other wavelets have two vanishing moments.}
\label{orthowavs}
\end{figure}

\subsection{The family $({}^{\delta\!}G,{}^{\delta\!}U_0)_{\delta \in \Delta}$} \label{Sconcrete3}
The index set $\vee_\Sigma$ is naturally identified with the set of `nodal dyadic' points, see Figure~\ref{fighb}, which is the natural index set for the one-dimensional hierarchical basis that we denote by $\{\phi_\lambda\colon \lambda \in \vee_\Sigma\}$.
   \begin{figure}
\begin{center}
\includegraphics{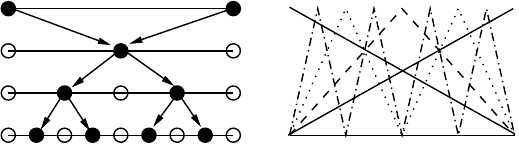}
\end{center}
\caption{Index set $\vee_\Sigma$ with parent-child relations, and the one-dimensional hierarchical basis.}
\label{fighb}
\end{figure}
Recalling that for $\delta \in \Delta$, $X^\delta=\Span\{\sigma_\lambda \otimes \phi_\nu \colon (\lambda,\nu) \in I_{\delta,0}=I_\delta\setminus(\vee_\Sigma\times \Gamma_D)\}$ for some lower set $I_\delta \subset \vee_\Sigma \times \mathfrak{N}$, we define
$$
{}^{\delta\!}G:=\Span\{\phi_\lambda \otimes \phi_\nu \colon (\lambda,\nu) \in I_\delta\},\,\,
{}^{\delta\!}U_0:=\Span\{\phi_\nu \colon (\lambda,\nu) \in I_\delta,\,\phi_\lambda(0)\neq 0\}.
$$
Since the level of resolution of these spaces is comparable to that of $X^\delta$, based on our experiences with wavelet and finite element methods we \emph{expect} that with this choice of $({}^{\delta\!}G,{}^{\delta\!}U_0)$ and the definition of $X^{\udelta\delta}$, that \emph{saturation} holds, i.e., that Assumption~\ref{saturation} assumption is valid.

Given $g \in Y'$ and $u_0 \in L_2(\Omega)$, it remains to define their approximations $({}^{\delta\!}g,{}^{\delta\!}u_0) \in ({}^{\delta\!}G,{}^{\delta\!}U_0)$. In general, the construction of these approximations depends on the data at hand. Below we give a construction that applies to general continuous $g$ and $u_0$, and that avoids quadrature issues.

For $\nu \in \mathfrak{N}$ with $\gen(\nu)=0$, let $\tilde{\phi}_\nu:=\delta_\nu$. Each $\nu \in \mathfrak{N}$ with $\gen(\nu)>0$ is the midpoint of an edge of a $T \in \tria$ with $\gen(T)=\gen(\nu)-1$. Denoting the endpoints of this edge as $\nu_1,\nu_2 \in \mathfrak{N}$, let $\tilde{\phi}_\nu:=\delta_\nu-\frac12(\delta_{\nu_1}+\delta_{\nu_2})$. Then $\{\tilde{\phi}_\nu\colon \nu \in \mathfrak{N}\} \subset C(\overline{\Omega})'$ is biorthogonal to $\{\phi_\nu\colon \nu \in \mathfrak{N}\}$.
With $\{\tilde{\phi}_\lambda\colon \lambda \in \vee_\Sigma\} \subset C(\overline{I})'$ defined analogously for the \new{temporal axis}, for $g \in C(\overline{I \times \Omega})$ and $u_0 \in C(\overline{\Omega})$ we define the interpolants
$$
{}^{\delta\!}g:=\sum_{(\lambda,\nu) \in I_\delta} (\tilde{\phi}_\lambda \otimes \tilde{\phi}_\nu)(g) \phi_\lambda \otimes \phi_\nu,\,\,
{}^{\delta\!}u_0:=\sum_{\{\nu\colon (\lambda,\nu) \in I_\delta,\,\phi_\lambda(0)\neq 0\}}   \tilde{\phi}_\nu(u_0)  \phi_\nu.
$$
Since we expect that for sufficiently smooth $g$ and $u_0$, the errors $\|g-{}^{\delta\!}g\|_{Y'}$ and $\|u_0-{}^{\delta\!}u_0\|_{L_2(\Omega)}$ are of higher order than the approximation error $\inf_{w \in X^\delta}\|u-w\|_X$, for our convenience in the adaptive Algorithm~\ref{54} we ignore errors caused by data-oscillation by setting $\eta(\cdot)\equiv0$.

Notice that setting up the matrix vector formulation of the system \eqref{52} that defines our approximation $u^\delta$
requires computing the vectors
$$
\big[\langle {}^{\delta\!}g,\psi_\mu \otimes \phi_\nu\rangle_{L_2(I \otimes \Omega)}\big]_{(\mu,\nu)\in I_{\udelta,0}^Y},\quad
\big[\langle {}^{\delta\!}u_0,\phi_\nu\rangle_{L_2(\Omega)}\big]_{\{\nu\colon (\lambda,\nu)\in I_{\delta,0},\,\sigma_\lambda(0)\neq 0\}}
$$
which can be performed in ${\mathcal O}(\dim X^\delta)$ operations because $I_\delta$ and $I_{\udelta,0}^Y$ are lower sets (and $\# I_{\udelta,0}^Y \lesssim \# I_\delta$).

\section{Numerical results}\label{Snumerics}
We test our algorithm on the heat equation, i.e.,~the parabolic problem with $a(t;\eta,\zeta)=\int_\Omega \nabla \eta \cdot \nabla \zeta d{\bf x}$, posed on a two-dimensional polygonal spatial domain $\Omega$, and Dirichlet boundary $\Gamma_D=\partial\Omega$.
Recall from \S\ref{Sconcrete2} the three-point continuous piecewise linear
temporal wavelet basis $\Sigma$, the orthonormal discontinuous piecewise linear
temporal wavelet basis $\Psi$, and the hierarchical continuous piecewise linear spatial basis $\Xi := \{\phi_\nu : \nu \in \mathfrak N_0\}$.

We consider `trial' spaces \new{$X^\delta$ spanned} by finite  subsets of $\Sigma \otimes \Xi$ whose index sets are lower sets (\new{cf.}~\eqref{eqn:Xdelta}--\eqref{eqn:Xdeltab}), and corresponding `test' spaces $Y^\delta$ spanned by finite subsets of $\Psi \otimes \Xi$ as defined in \eqref{eqn:Ydelta}--\eqref{eqn:Ydeltab}.
We construct the enlarged trial space
$X^{\udelta}$ as defined in Def.~\ref{def:Xudelta}, with corresponding test space $Y^{\udelta}$.

For a given level $N \in \N$, $\Span \{\sigma_\lambda : |\lambda| \leq N\}$ coincides with the span of the
continuous piecewise linears on an $N$-times recursive dyadic refinement of $I$, and
$\Span \{\phi_\nu \in \Xi : \gen(\nu) \leq 2N\}$ coincides with that of the
continuous piecewise linears, zero at $\partial\Omega$, on a $2N$-times recursive bisection refinement of an initial partition $\tria_\perp$.
Therefore, the span of the \emph{`full'} tensor product
$\{\sigma_\lambda \colon |\lambda| \leq N\} \otimes \{\phi_\nu \colon \gen(\nu) \leq 2N\}$
equals a space of lowest order continuous finite elements w.r.t.~a quasi-uniform shape regular product mesh  into prismatic elements.

Taking only those index pairs $(\lambda, \nu)$ for which
$2 |\lambda| + \gen(\nu) \leq 2N$ produces a \emph{`sparse'} tensor product on level $N$.
Sparse tensor \new{products overcome} the \emph{curse of dimensionality} in the sense that for smooth solutions they achieve
a rate in $X$-norm \new{equal to} the best rate in $H^1(\Omega)$-norm that can
be expected for the corresponding stationary problem on the spatial domain, here the Poisson equation; \new{cf.~}Sect.~\ref{sec:rates}.

We run our adaptive Algorithm~\ref{54} with $\theta = 0.5$ and $\xi = \tfrac{1}{2}$, computing ${}^{\delta\!}g$ and ${}^{\delta\!} u_0$ as in Sect.~\ref{Sconcrete3}.
Since we envisage that in our experiments data-oscillation errors are not dominant, for our convenience we took $\omega=\infty$.
We solve the arising linear system of~\eqref{52} using Preconditioned CG, using the previous solution as initial guess.
We then perform D\"orfler marking on the residual, yielding a minimal set $J$, and finally choose $I_{\tilde \delta}$ as the smallest lower set containing $J \cup I_\delta$.
Due to this constraint generally we add index pairs outside of the marked set, i.e.~$I_{\tilde \delta} \setminus I_{\delta} \supsetneq J$. Still, in our experiments, we \emph{observe} $\#I_{\tilde \delta} - \# I_{\delta} \lesssim \# J$ with a moderate constant.

\begin{remark} Rather we would have applied an algorithm that produces a $I_{\tilde \delta}$ such that
$I_{\tilde \delta} \setminus I_\delta$  is \emph{guaranteed} to have an, up to a multiplicative factor, smallest cardinality among all
lower sets $I_{\tilde \delta} \supset I_\delta$ that realize the bulk criterion. Such an algorithm was introduced in \cite{20.5,22.58} for `single-tree' approximation, but seems not to be available for the `double-tree' (i.e.~lower set) constraint that we need here.
\end{remark}

We compare adaptive refinement with non-adaptive
full- and sparse tensor products, and monitor the \new{$X$-norm} error estimator from
Proposition~\ref{alternative}, the residual error estimator from
Proposition~\ref{m18}, and the $L_2(\Omega)$ trace error at ${t=0}$.


\subsection{Condition numbers of preconditioner}
For the calibration of our preconditioners, we consider $\Omega := \new{(0,1)}^2$, and
compare \emph{uniformly refined} space-time meshes with \emph{locally refined} meshes with refinements towards $\{0\} \times \partial \Omega$.

The replacement of the nonlocal operator $({E_Y^\udelta}' A E_Y^\udelta)^{-1}$ in the
forward application of $S^{\udelta \delta}$ by the block-diagonal preconditioner $K_Y^\udelta$
from Sect.~\ref{sec:precX} is only guaranteed to result
in a convergent algorithm when the eigenvalues of $K_Y^\udelta {E_Y^\udelta}' A E_Y^\udelta$ are sufficiently close to one.

In Table~\ref{table:DY-precond}, we investigate the values
$\kappa_\delta := \max\{\lambda_{\max}({\bf K}_Y^\udelta {\bf A}_Y^\udelta), 1/\lambda_{\min}({\bf K}_Y^\udelta {\bf A}_Y^\udelta)\}$
with ${\bm A}_Y^\udelta$ the matrix representation of ${E_Y^\udelta} ' A E_Y^\udelta$,
and ${\bf K}_Y^\udelta$ built from spatial multigrid preconditioners ${\bf K}_\mu^\udelta$ corresponding to
$n$ V-cycles.
In each V-cycle we applied one pre- and one post Gauss--Seidel smoother. In case of a locally refined spatial mesh, on each level these Gauss--Seidel updates were restricted to the vertices whose generation is equal to that level as well as both endpoints of the edge on which these vertices were inserted (\cite{316.55}).
We see that for both uniform and locally refined space-time meshes, $\kappa_\delta$
converges to $1$ rapidly in $n$, and is essentially independent of $\dim X^\delta$.
In our examples,  $\kappa_\delta$ is \new{close enough} to one already for $n=1$.

\begin{table}[b]
  \begin{tabular}{rrrrrrrrrrr}\toprule
    $\dim X^\delta$ & $n=1$ & $n=2$ & $n=3$ & $n=4$ & $n=5$  & $n=6$\\\midrule
    \emph{uniform} \quad
    729     & 1.343  & 1.070  & 1.017  & 1.004  & 1.001  & 1.000 \\
    35937   & 1.360   & 1.075  & 1.019  & 1.004  & 1.001  & 1.000 \\
    2146689  & 1.365  & 1.077    & 1.019  & 1.004  & 1.001  & 1.000\\\midrule
    \emph{local} \quad
766         & 1.306 & 1.058 & 1.013 & 1.003 & 1.001 & 1.000 \\
30151       & 1.307 & 1.058 & 1.013 & 1.003 & 1.001 & 1.000 \\
1964797     & 1.307 & 1.058 & 1.013 & 1.003 & 1.001 & 1.000 \\
    \bottomrule
  \end{tabular}
  \caption{Values $\kappa_\delta := \max\{\lambda_{\max}({\bf K}_Y^\udelta {\bf A}_Y^\udelta), 1/\lambda_{\min}({\bf K}_Y^\udelta {\bf A}_Y^\udelta)\}$ using spatial multigrid with $n$ V-cycles.}
     \label{table:DY-precond}
\end{table}

Fixing $n=1$ for the forward application of $S^{\udelta \delta}$, we want to precondition
$S^{\udelta \delta}$ itself as well. Following Sect.~\ref{sec:precX},
we build a block-diagonal preconditioner taking ${\bf K}_\lambda^\delta$ to correspond to $m$ V-cycles of the aforementioned multigrid method now applied to ${\bf A}_\lambda^\delta+2^{|\lambda|} {\bf M}_\lambda^\delta$ with ${\bf A}_\lambda^\delta$ and ${\bf M}_\lambda^\delta$ being stiffness- \new{and} mass-matrices.
Table~\ref{table:DX-precond} shows the condition numbers of the preconditioned
matrix. We again see fast stabilization in $m$ as well as in $\dim X^\delta$.
We fix $m=3$ in the sequel. Most interestingly, in every of our example problems,
\new{only one or two PCG iterations suffice to fulfill the stopping criterion $\tilde{t}_\delta \leq \frac{t_\delta}{2}$ on the algebraic error in Algorithm~\ref{54}.}

\begin{table}[b]
  \begin{tabular}{rrrrrrrrrrr}\toprule
    $\dim X^\delta$ & $m=1$ & $m=2$ & $m=3$ & $m=4$ & $m=5$ & $m=6$\\\midrule
    \emph{uniform} \quad
    4913        & 9.196 & 6.119 & 6.048 & 6.042 & 6.041 & 6.041 \\
    35937       & 9.718 & 6.315 & 6.263 & 6.260 & 6.260 & 6.260 \\
    274625      & 9.991 & 6.750 & 6.749 & 6.751 & 6.752 & 6.752 \\
    2146689  & 10.115  & 7.080  & 7.087  & 7.088  & 7.088  & 7.088 \\
    \midrule
    \emph{local} \quad
3520        & 5.707 & 5.132 & 5.110 & 5.111 & 5.111 & 5.111 \\
30151       & 6.355 & 5.734 & 5.706 & 5.704 & 5.704 & 5.704 \\
244870      & 7.619 & 6.879 & 6.843 & 6.841 & 6.841 & 6.841 \\
1964797     & 9.353 & 8.734 & 8.703 & 8.701 & 8.701 &   8.701 \\
    \bottomrule
  \end{tabular}
  \caption{Spectral condition numbers of $K_X^\delta S^{\udelta \delta}$, using spatial multigrid with $m$ V-cycles.}
     \label{table:DX-precond}
\end{table}

\subsection{Smooth problem}
We consider the square domain $\Omega := \new{(0,1)}^2$ and prescribe
\[
  u(t,x,y) := (1 + t^2) x (1-x) y (1-y)
\]
with derived data $u_0$ and $g$. For this smooth solution, full and sparse tensor products are expected to yield the best possible error decays proportional to $(\dim X^\delta)^{-1/3}$ and $(\dim X^\delta)^{-1/2}$, respectively.

The left side of Figure~\ref{fig:smooth} shows the error progressions for the smooth problem. We
plot the error estimator \new{$\|{}^{\delta\!}u - \tilde u^\delta\|_{X^{\udelta\infty}} \eqsim \|{}^{\delta\!}u - \tilde u^\delta\|_X$} from Proposition~\ref{alternative},
the residual error estimator $\|\mathbf r^{\delta}\|$,
and $\|\gamma_0 ({}^{\delta\!}u - \tilde u^\delta)\|_{L_2(\Omega)}$.
We see that the error progressions are as expected. For this solution,
adaptive refinement yields no advantage over sparse grid refinement.
We observe a higher order of convergence for the trace at $t=0$ measured in $L_2(\Omega)$.

\begin{figure}
\includegraphics[width=0.49\linewidth]{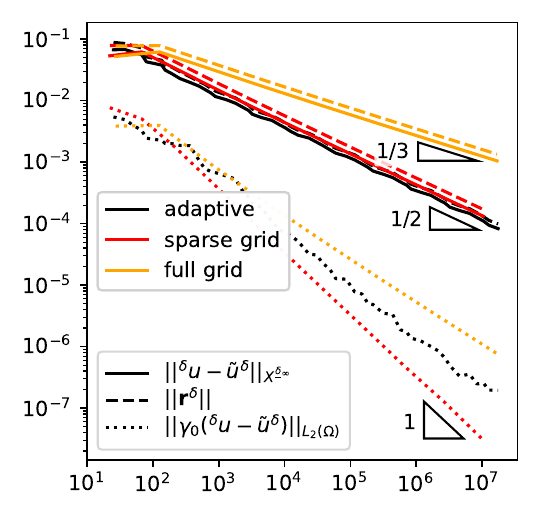}
\includegraphics[width=0.49\linewidth]{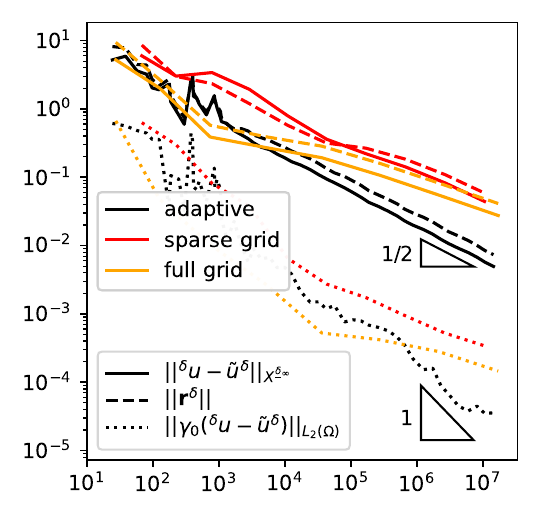}
\vspace{-1em}
\caption{Error progressions for (left) the \emph{smooth problem} and (right) the \emph{moving peak} problem. Shown: estimated $X$-norm error  (solid line),
 residual norm (dashed), and $t=0$ trace error (dotted) as a function
of $\dim X^\delta$ for adaptive (black), sparse grid (red), and full grid refinement (orange).}
\label{fig:smooth}
\end{figure}

\subsection{Moving peak problem {\cite{169.052}}}
Again for $\Omega := \new{(0,1)}^2$, we select
\[
    u(t,x,y) := x(1-x)y(1-y) \exp(-100 [(x-t)^2 + (y-t)^2]).
\]
The solution is smooth, and almost zero everywhere except on a small strip near the diagonal from $(0,0,0)$ to $(1,1,1)$ of the space-time cylinder. As $u$ is smooth, we expect sparse grid refinements to asymptotically yield the optimal error decay proportional to $(\dim X^\delta)^{-1/2}$, albeit with a terrible constant. Adaptive refinement should be able to achieve the same rate at quantitatively smaller doubletrees.

\new{The} right of Figure~\ref{fig:smooth} \new{shows} that the sparse grid rate is not (yet) optimal, while our adaptive routine is able to find the optimal rate from $\dim X^\delta \approx 10^3$ onwards. Figure~\ref{fig:moving-peak2} shows the number of basis functions $\sigma_\lambda \otimes \phi_\nu$ whose supports intersect given points in the time-space cylinder. We see the adaptation to the moving peak.

\subsection{Cylinder problem}
Selecting the L-shaped domain $\Omega :=  \new{(-1,1)^2 \setminus (-1,0]^2}$ with data
$u_0 = 0$ and $g(t,x,y) := t\cdot \1_{\{x^2 + y^2 < 1/4\}}$, the true
solution is known to be singular at the re-entrant corner and at the wall of the cylinder $\{(t,x,y)\colon x^2 + y^2 = 1/4\}$. We took this example from \cite{75.257}.
The left side of Figure~\ref{fig:singular} shows the error progression for this cylinder problem.
We see that the full grid error decay proportional to $(\dim X^\delta)^{-1/4}$ is improved to \new{$(\dim X^\delta)^{-1/3}$} by considering sparse grids. Adaptive refinement, however, achieves the best possible error decay proportional to $(\dim X^\delta)^{-1/2}$, recovering the rate for a smooth solution.

\begin{figure}
\centering
\includegraphics[trim={2.5cm 1.0cm 1.5cm 0.5cm},clip,width=\linewidth]{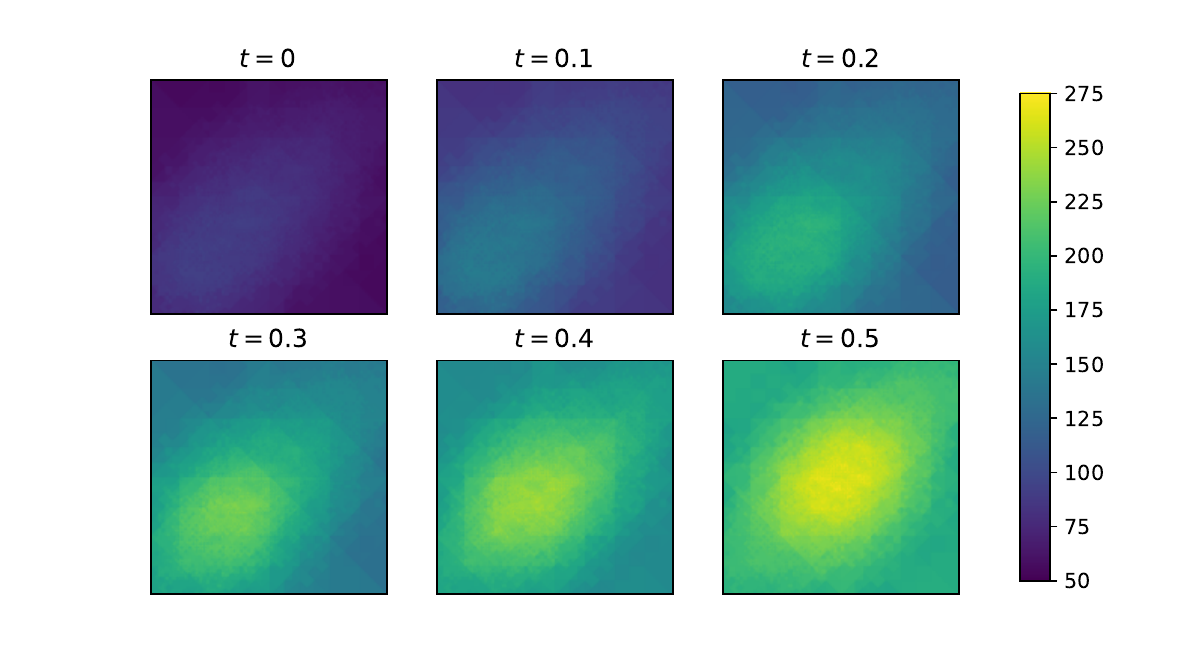}
\caption{\emph{Moving peak} problem, adaptive lower set with $\dim X^\delta = 89\,401$. Shown: $\#\{(\lambda, \nu) \in I_\delta \colon (t,x,y) \in \supp \sigma_\lambda\otimes \phi_\nu\}$ for a selection of times $t$.}
  \label{fig:moving-peak2}
\end{figure}

\subsection{Singular problem}
\new{Again for $\Omega := \new{(-1,1)^2 \setminus (-1,0]^2}$, we select data $u_0 \new{=} \1$,
and $g = 0$}. The solution has a strong singularity  along $\{0\} \times \partial \Omega$ due to the incompatibility of initial- and boundary conditions, in addition to the singularity at the re-entrant corner $(0,0)$. At the right of Figure~\ref{fig:singular},
for uniform refinement, we see the extremely slow error decay proportional to $(\dim X^\delta)^{-1/11}$, already found in \cite{75.257}. Interestingly, sparse grid refinement offers no rate improvement over full grid refinement.
The adaptive algorithm yields a much better error decay proportional to $(\dim X^\delta)^{-2/5}$.
We observed that increasing the D\"orfler marking parameter to $\theta = 0.7$  decreases the convergence rate to $-1/3$, whereas a $\theta$ smaller than $0.5$ did not improve the rate beyond $-2/5$.
Looking at Figure~\ref{fig:singular-centers}, we see strong adaptivity towards
$\{0\} \times \partial \Omega$ and $I \times \{(0,0)\}$, \new{even seeing} basis functions $\sigma_\lambda \otimes \phi_\nu$ \new{of} $X^\delta$ whose barycenter is at $t=2^{-14} \approx 10^{-4}$.

\begin{figure}
\centering
\includegraphics[width=0.49\linewidth]{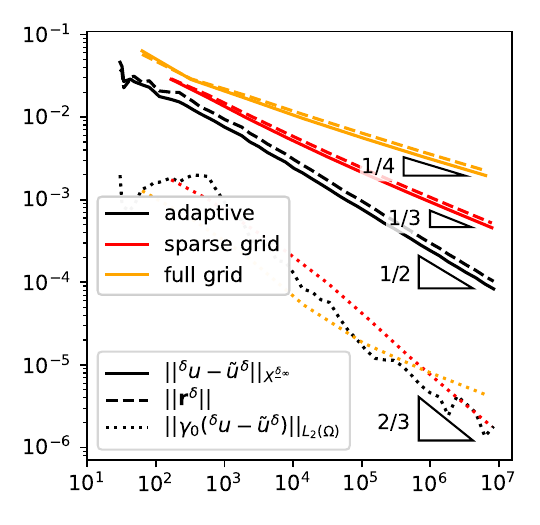}
\includegraphics[width=0.49\linewidth]{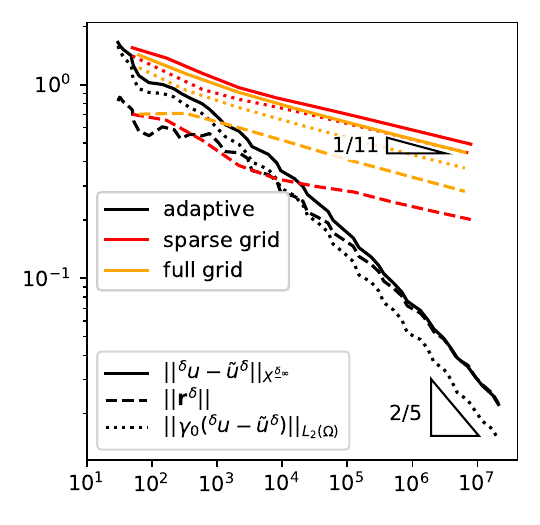}
\vspace{-1em}
  \caption{Error progressions for (left) the \emph{cylinder} problem and (right) the \emph{singular} problem. Shown: estimated $X$-norm error  (solid line),
 residual norm (dashed), and $t=0$ trace error (dotted) as a function
of $\dim X^\delta$ for adaptive (black), sparse grid (red), and full grid refinement (orange).}
  \label{fig:singular}
\end{figure}
\begin{figure}
\includegraphics[width=0.49\linewidth]{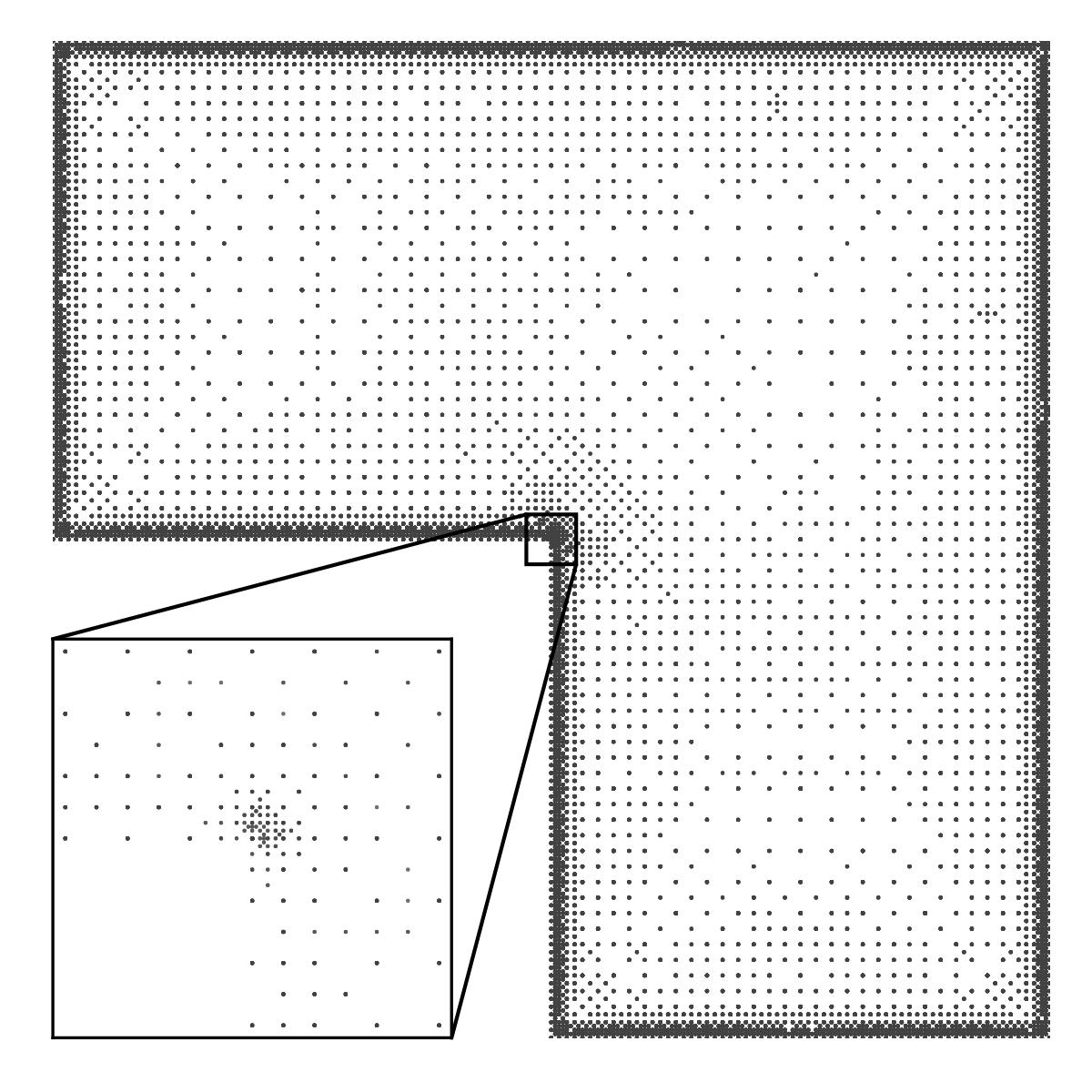}
\hfill
\includegraphics[width=0.49\linewidth]{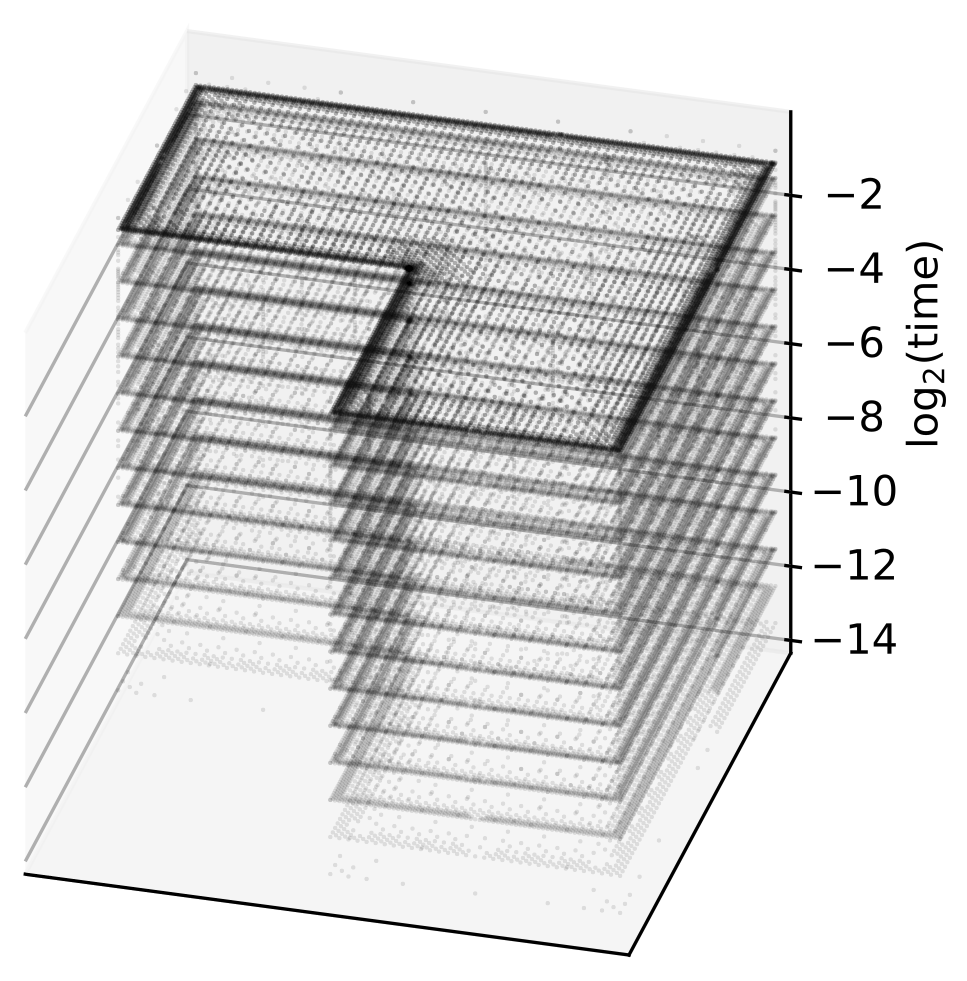}
\caption{Barycenters of supports of basis functions $\sigma_\lambda \otimes \phi_\nu$ spanning $X^\delta$  generated by Algorithm~\ref{54} of dimension 81\,074 for the \emph{singular} problem. Left: a top-down view, with a 10$\times$ zoom to the origin; right: centers in spacetime, logarithmic in time.}
\label{fig:singular-centers}
\end{figure}

\subsection{Gradedness and error reduction}
In Sect.~\ref{sec:adaptive} we used~\eqref{m17} to demonstrate proportionality of
$\|{\bf r}^\delta\|$ and $\|u - u^\delta\|_X$, as well as a constant error reduction
in each iteration of the adaptive algorithm. In Proposition~\ref{stableset}, we
showed that~\eqref{m17} holds when the gradedness $L_\delta$ of
Definition~\ref{def:gradedness} is uniformly bounded.

In the left picture of Figure~\ref{fig:gradedness}, we see however a more than expected increase
in gradedness, where in particular for the singular problem we observe a logarithmic increase in terms of $\dim X^\delta$. However, this turns out not to be a problem in
practice: Figures~\ref{fig:smooth} and~\ref{fig:singular} demonstrate that the
residual error $\|{\bf r}^\delta\|$ and the estimated $X$-norm error
$\|{}^{\delta\!}u - \tilde u^\delta\|_{X^{\udelta\infty}}$ are very close, and even converge for the
singular problem. Moreover, in the right picture of Figure~\ref{fig:gradedness}, we see
a constant error reduction of $\breve \rho \approx 0.89$ at every step of
the adaptive algorithm, and hence, that the conclusion of Theorem~\ref{thm:main} holds in practice.
\begin{figure}
    \centering
    \includegraphics[width=\linewidth]{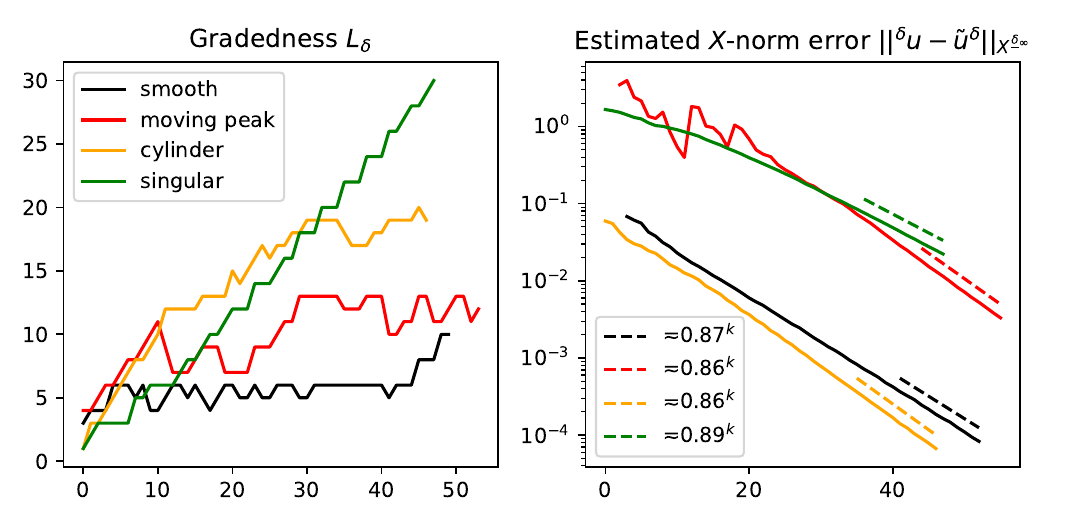}
    \vspace{-2em}
    \caption{Gradedness and estimated $X$-norm error at every iteration of the adaptive loop, for the four different model problems under consideration.}
    \label{fig:gradedness}
\end{figure}

\subsection{Total runtime and memory consumption}
Figure~\ref{fig:constant} shows the total runtime and peak memory consumption after every iteration of the adaptive algorithm. The top row shows absolute values, and the bottom row values relative to $\dim X^\delta$.

The left of the figure shows that the adaptive algorithm runs in optimal linear time in the dimension of the current trial space.

The right of the figure shows that the peak memory is linear as well, stabilizing to around 15kB per degree of freedom.
This is relatively high, mainly because our implementation uses trees rather than hash maps to represent vectors to ensure a linear-time implementation of the matrix-vector products (cf.~Rem.~\ref{rem:matvec}).

\begin{figure}
    \centering
    \includegraphics[width=\linewidth]{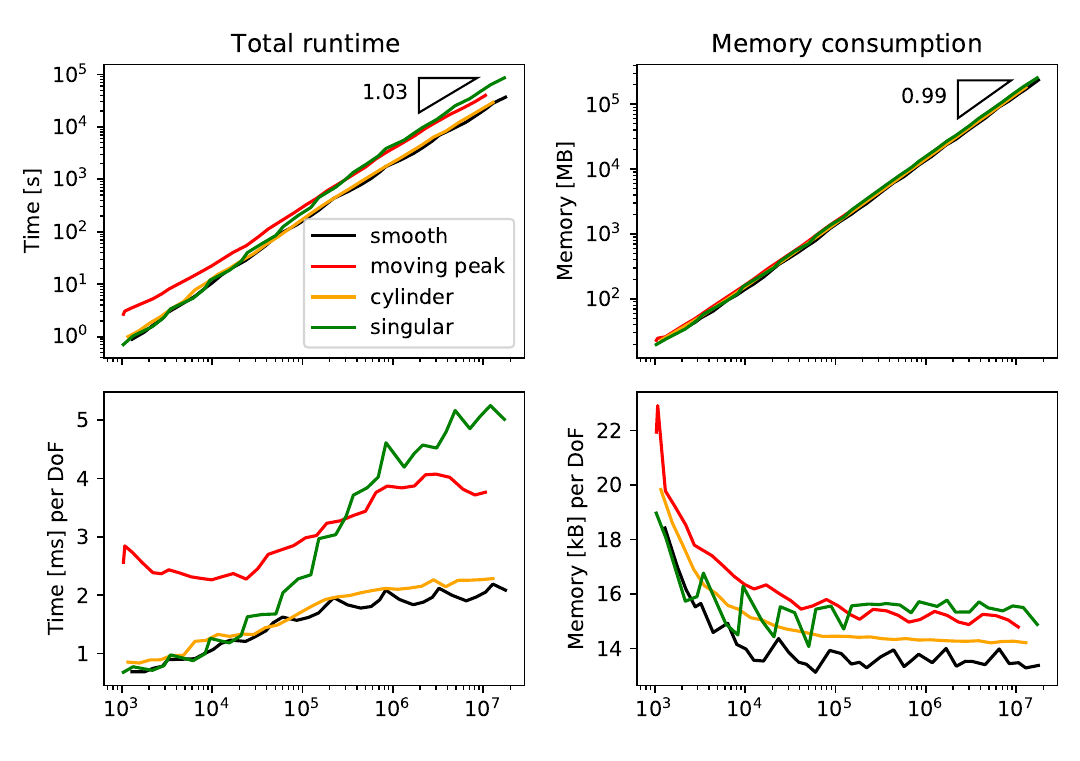}
    \vspace{-2em}
    \caption{Total runtime and peak memory consumption as function of $\dim X^\delta$, measured after every iteration of the adaptive loop, for the four different model problems.}
    \label{fig:constant}
\end{figure}

\section{Conclusion} \label{Sconclusion}
\new{We constructed} an adaptive solver for a space-time variational formulation of parabolic evolution problems.
The collection of trial spaces are given by the spans of sets of tensor products of wavelets-in-time and hierarchical basis functions-in-space.
Compared to our previous works \cite{38.4,243.867} where we employed `true' wavelets also in space, the theoretical results are weaker.
\new{We demonstrated} $r$-linear convergence of the adaptive routine, but have not shown optimal rates at linear complexity. On the other hand, the \new{runtimes we} obtained with the current approach are much better.

\newcommand{\etalchar}[1]{$^{#1}$}

\end{document}